\documentclass{amsart}
\usepackage{graphicx} % Required for inserting images
\usepackage{hyperref}
\usepackage{aliascnt}
\usepackage{bm}
\usepackage{stmaryrd}
\usepackage{amsmath}
\usepackage{mathtools}
\usepackage{tikz-cd}
\usepackage{cleveref}
\usepackage[english]{babel}
\usepackage{csquotes}
\title[Transcendence of multivariate Mahler functions]{Transcendence of multivariate Mahler functions and algebraic relations between their values}
\author{Enzo Brechler}
\date{February 2026}

\theoremstyle{plain}
\newtheorem{theo}{Theorem}[section]

\newaliascnt{lem}{theo}
\newtheorem{lem}[lem]{Lemma}
\aliascntresetthe{lem}

\newaliascnt{prop}{theo}
\newtheorem{prop}[prop]{Proposition}
\aliascntresetthe{prop}

\newaliascnt{coro}{theo}
\newtheorem{coro}[coro]{Corollary}
\aliascntresetthe{coro}

\theoremstyle{definition}

\newaliascnt{defi}{theo}
\newtheorem{defi}[defi]{Definition}
\aliascntresetthe{defi}

\newaliascnt{ex}{theo}
\newtheorem{ex}[ex]{Example}
\aliascntresetthe{ex}

\newaliascnt{rk}{theo}
\newtheorem{rk}[rk]{Remark}
\aliascntresetthe{rk}

\newaliascnt{nt}{theo}
\newtheorem{nt}[nt]{Notation}
\aliascntresetthe{nt}

\crefname{theo}{theorem}{theorems}
\Crefname{theo}{Theorem}{Theorems}
\crefname{lem}{lemma}{lemmas}
\Crefname{lem}{Lemma}{Lemmas}

\crefname{prop}{proposition}{propositions}
\Crefname{prop}{Proposition}{Propositions}

\crefname{defi}{definition}{definitions}
\Crefname{defi}{Definition}{Definitions}

\crefname{ex}{example}{examples}
\Crefname{ex}{Example}{Examples}

\crefname{rk}{remark}{remarks}
\Crefname{rk}{Remark}{Remarks}

\crefname{coro}{corollary}{corollaries}
\Crefname{coro}{Corollary}{Corollaries}

\crefname{nt}{notation}{notations}
\Crefname{nt}{Notation}{Notations}

\newcommand{\C}{\mathbb{C}}
\newcommand{\R}{\mathbb{R}}
\newcommand{\Q}{\mathbb{Q}}
\newcommand{\Z}{\mathbb{Z}}
\newcommand{\N}{\mathbb{N}}
\newcommand{\K}{\mathbb{K}}

\newcommand{\A}{\overline{\mathbb{Q}}}
\newcommand{\z}{\bm{z}}
\newcommand{\vaset}{\llbracket     1,n \rrbracket}
\newcommand{\al}{\bm{\alpha}}
\newcommand{\GL}{\operatorname{GL}}
\newcommand{\Frac}{\operatorname{Frac}}
\newcommand{\tr}{\operatorname{tr.deg}}
\newcommand{\Span}{\operatorname{Span}}
\newcommand{\Diag}{\operatorname{Diag}}
\newcommand{\M}{\operatorname{M}}
\newcommand{\len}{\operatorname{len}}

\begin{document}

\begin{abstract}
We study the functional properties of Mahler functions of several variables associated with a large class of Mahler transformations. We prove that these functions are always analytic on a neighborhood of zero, and meromorphic on the open unit ball in $\C^n$, as well as in $\C_p^n$ for all primes $p$. This enables us to prove a rational-transcendental dichotomy for these functions. As a consequence, we strengthen the lifting theorem recently obtained by Adamczewski and Faverjon. This theorem allows us to systematically compute the algebraic relations between values of multivariate Mahler functions related by a Mahler system, given the algebraic relations between the functions themselves. Finally we prove a multivariate descent theorem in the case of regular points.
\end{abstract}
\maketitle
\tableofcontents

\section{Introduction}
Given convergent power series $f_1(z_1,\ldots,z_n), \ldots,f_m(z_1,\ldots,z_n) \in \A\{z_1,\ldots,z_n\}$ and an algebraic point $\bm{\alpha}=(\alpha_1,\ldots,\alpha_n) \in (\A^*)^n$ lying in their domain of convergence, a fundamental question in transcendence theory is to understand the algebraic relations between the numbers $f_1(\al), \ldots, f_m(\al)$. In particular, we would like to know whether the $f_i(\al)$ are algebraic or transcendental.

In the univariate case $n=1$, this problem is well understood for two famous classes of functions: the Siegel $E$-functions, which satisfy linear differential equations, and the Mahler $M$-functions, which satisfy linear difference equations associated to a Mahler operator $z \longrightarrow z^q$ with $q \geq 2$.

In this paper, we will focus on the case of Mahler functions in the multivariate setting. We let $T=(t_{ij})$ be a non-singular matrix with nonnegative integer coefficients. Such a matrix $T$ defines a transformation of $\C^n$ via $$T\z=(z_1^{t_{11}}z_2^{t_{21}} \ldots z_n^{t_{n1}}, \ldots ,z_1^{t_{1n}}z_2^{t_{2n}} \ldots z_n^{t_{nn}}) \, .$$
$T$ is then extended to a difference operator on functions (polynomials, power series...) by $(Tf)(\z)=f(T\z)$.\\
Equivalently, we can represent this as a substitution rule:
\begin{align*}
    z_1 &\rightarrow z_1^{t_{11}}z_2^{t_{21}} \ldots z_n^{t_{n1}}& \\
    &\vdots&\\
    z_n &\rightarrow z_1^{t_{1n}}z_2^{t_{2n}} \ldots z_n^{t_{nn}} \, .
\end{align*}
An $M_T$-function is then defined as a formal power series $f \in \A[[\z]]$ satisfying a functional equation of the form
$$\sum_{i=0}^m{p_i(\z)f(T^i\z)}=0$$
where the $p_i$ are polynomials with algebraic coefficients (and the $p_i$'s are not all zero).\\
To study $M_T$-functions, it is more convenient to work with so-called $T$-Mahler systems instead, which are systems of the form $Y(\z)=A(\z)Y(T\z)$ where $A \in \GL_m(\A(\z))$. It is a standard result in the univariate case that a power series $f$ is an $M_T$-function if and only if it is a coordinate of a solution of a Mahler system, and we will show that this remains valid in the multivariate case. If $(f_1(\z),\ldots,f_m(\z))^*$ (throughout this paper, the transpose will be denoted by a star to avoid confusion with the operator $T$) is a solution of a $T$-Mahler system, we would like to obtain transcendence results on the numbers $f_i(\al)$  where $\al \in (\A^*)^n$ is an algebraic point.\\
It is well known that some assumptions are necessary on the transformation $T$ and on $\al$: $\al$ must be regular with respect to the system, meaning that for all $k \in \N$, $A(T^k\al)$ is well defined and invertible, and furthermore, $(T,\al)$ must be admissible in the following sense.

\begin{defi} \label{defi:admi}
Let $T$ be a transformation and $\al \in (\A^*)^n$ an algebraic point. We say that $(T,\al)$ is admissible if the following conditions are satisfied.
\begin{enumerate}
    \item $T$ belongs to the class $\mathcal{T}$, which means that $T$ is non-singular, that none of the eigenvalues of $T$ is a root of unity, and that $T$ has a Perron-Frobenius eigenvector (an eigenvector with positive coordinates associated to the dominant eigenvalue of $T$).
    \item $T^k\al \longrightarrow\bm{0}$.
    \item $\al$ is $T$-independent, which means that there is no non-zero vector of integers $\bm{\mu}=(\mu_1,\ldots,\mu_n)$ such that $(T^k\al)^{\bm{\mu}}=1$ for all $k$ in an infinite arithmetic progression (see \Cref{not:1} for the monomial notation). Note that if the coordinates of $\al$ are multiplicatively independent, then $\al$ is $T$-independent with respect to every non-singular transformation $T$.    
\end{enumerate}
\end{defi}

The transcendence theory for Mahler functions originated with three seminal papers by Mahler \cite{mahler1929arithmetische, mahler1930arithmetische, mahler1930uber} in 1929. After having been largely forgotten for about fifty years, Mahler's method experienced a revival in the 1970s and 1980s with important contributions by Loxton and Van der Poorten \cite{LVDP1, LVDP2,LVDP3}, Kubota \cite{Kubota}, Nishioka \cite{Nish1983} and Masser \cite{Masser}.

In 1990, Nishioka \cite{Nish} proved a landmark result for the univariate case $n=1$ ($T=(q)$ with $q \ge 2$). It states that if $Y(z)=(f_1(z),\ldots,f_m(z))^*$ is a solution of a $q$-Mahler system and if $\alpha \in \A^*$ (with $|\alpha|<1$) is a regular point with respect to this system, then we have the equality of transcendence degrees
$$\tr_{\A}(f_1(\alpha),\ldots,f_m(\alpha))=\tr_{\A(z)}(f_1(z),\ldots,f_m(z)) \, .$$
This is the analog of the famous Siegel-Shidlovskii theorem for $E$-functions. We also refer to the book \cite{nishioka2006mahler} for an overview of Mahler's method until 1996.\\
In 2015, Nishioka's result was improved by Philippon \cite{Phil}, and Adamczewski and Faverjon \cite{adamczewski2017methode} who proved the so-called Lifting theorem, an analog of the result obtained by Beukers in 2006 in the case of $E$-functions \cite{Beu}. These lifting theorems state that (under the same conditions as in Nishioka's theorem), every homogeneous algebraic relation over $\A$ between our values $f_1(\alpha), \ldots, f_m(\alpha)$ is obtained from a homogeneous algebraic relation over $\A(z)$ between the functions $f_1(z), \ldots, f_m(z)$, after evaluating at $\alpha$.\\
After these lifting theorems were established, the transcendence theory for $E$-functions and $M_q$-functions became very similar, as shown in \cite{E+M}. One of the main consequences of the lifting theorem is the descent theorem (see \cite{adamczewski2017methode}), stating that if $f_1, \ldots, f_m$ are $M_q$-functions with coefficients in a number field $\K$ and $\alpha \in \K$, then the linear dependence of $f_1(\alpha),\ldots,f_m(\alpha)$ over $\A$ implies their linear dependence over $\K$. An important motivation for these results comes from a connection between automatic numbers and Mahler functions discovered by Cobham \cite{cobham1968hartmanis}: the lifting theorem implies that a real automatic number is either rational or transcendental (a result which was conjectured by Cobham in his paper). Finally, Adamczewski and Faverjon \cite{Algotrans} showed how to use the lifting theorem to devise an algorithm to systematically solve our initial problem of finding the algebraic relations between values of univariate Mahler functions.

In the multivariate case, a significant breakthrough was achieved more recently by Adamczewski and Faverjon (Theorem 3.3 in \cite{adamczewski2020mahler}).  This is the theorem that this paper aims to refine in order to obtain the following \enquote{optimal lifting theorem}.
\begin{theo} [Lifting Theorem] \label{thm:lift} $\newline$
Let $Y=(f_1,\ldots,f_m)^*$ be a vector of power series in $\A[[\z]]$ which satisfies a system $Y(\z)=A(\z)Y(T\z)$ with $A \in \GL_m(\A(\z))$. We assume that $\al \in (\A^*)^n$ is a regular point for this system, and that $(T,\al)$ is admissible.\\
Then for any homogeneous polynomial $P \in \A[X_1,\ldots,X_m]$ such that
$$P(f_1(\al),\ldots,f_m(\al))=0$$
there exists a polynomial $Q \in \A[\z,X_1,\ldots,X_m]$, homogeneous in the variables $X_i$, such that
$$Q(f_1(\z),\ldots,f_m(\z))=0 \text{\ and \ } Q(\al,X_1,\ldots,X_m)=P(X_1,\ldots,X_m) \, .$$
\end{theo}
Adamczewski and Faverjon obtained a weaker version of \Cref{thm:lift} in \cite{adamczewski2020mahler}, where $Q \in R[X_1,\ldots,X_m]$ with the ring $R$ defined as the algebraic closure of $\A(\z)$ in $\A\{\z-\al\}$. Replacing $R$ by the much more satisfying $\A[\z]$ can be seen as the main goal of this paper. Adamczewski and Faverjon also obtained the conclusion of \Cref{thm:lift} under the condition that $\A(\z,f_1(\z),\ldots,f_m(\z))$ is a regular field extension of $\A(\z)$ (meaning that every element of $\A(\z,f_1(\z),\ldots,f_m(\z))$ is either in $\A(\z)$ or transcendental over $\A(\z)$). It will be shown in this paper that this regularity condition always holds and hence can be dropped.\\
A direct consequence of the lifting theorem, obtained when taking $P$ homogeneous of degree $1$, is the following linear independence result.
 \begin{coro} \label{prop:indep}
Under the assumptions of the lifting theorem, if the functions \linebreak $f_1,\ldots,f_m$ are linearly independent over $\A(\z)$, then the values $f_1(\al),\ldots,f_m(\al)$ are linearly independent over $\A$.
\end{coro}

Given $f$ an $M_T$-function and $\al$ an admissible point, the previous result gives the following effective criterion to prove that $f(\al)$ is transcendental.

\begin{coro} \label{coro:trans} 
Let $f$ be an $M_T$-function and let
$$\sum_{i=0}^m{p_i(\z)f(T^i\z)}=p_{-1}(\z)$$
be the minimal inhomogeneous equation satisfied by $f$ (meaning with minimal order $m$). Assume that $\al$ is regular with respect to this equation (meaning that $p_0(T^k\al)p_m(T^k\al) \neq 0$ for all $k \in \N$) and assume that $(T,\al)$ is admissible. Then $f(\al)$ is transcendental.
\end{coro}

The study of multivariate Mahler functions has applications to certain classes of univariate power series, such as Hecke-Mahler series and generating functions of morphic sequences. Indeed, the values of these series at algebraic points can be seen as values of an auxiliary multivariate Mahler function. Consequently, the lifting theorem and its consequences apply to these values. We plan to return to this question in future work.\\

A large part of this paper is devoted to proving that the regularity condition always holds. This will be done by studying functional properties of $M_T$-functions. The algebraic properties, as well as the analytic properties of these functions, are well known in the univariate case $n=1$. For example, we know that for $T=(q)$ where $q\ge 2$, a non-rational $M_q$-function $f$ is always meromorphic in the open unit disk, and that the unit circle is a natural boundary for $f$ (Theorem 4.3 in \cite{Rand}). This implies that an $M_q$-function is either a rational function or transcendental over $\A(\z)$. We shall refer to this property as the functional rational-transcendental dichotomy.\\
Some of the main properties of Mahler functions and of Mahler equations were studied by Dumas in his PhD thesis \cite{Dumas}. In more recent works, we can mention algorithmic methods to find solutions of Mahler equations (see \cite{CompSol} and \cite{CompHahn}), a classification of $M_q$-functions by the height of their coefficients \cite{Gap}, as well as a result showing that a non-rational $M_q$-function is always hypertranscendental over $\A(z)$ (see \cite{Hypertrans}).\\
By contrast, the functional properties of $M_T$-functions remain much less understood in the multivariate setting. Another goal of this paper is to generalize some of the classical functional properties that we listed above to the multivariate setting.\\

For the transcendence results on values of Mahler functions (such as the lifting theorem), we need $T$ to be in the class $\mathcal{T}$, as described above. However, if we are only studying Mahler functions at a functional level, we can use weaker conditions. Therefore, we introduce a new class of transformations which we will call the class $\mathcal{F}$.

\begin{defi} \label{defi:f}
We say that $T$ belongs to the class $\mathcal{F}$ if $T$ is non-singular and if there is no $i \in \vaset$ and $k \in \N^*$ such that $T^kz_i=z_i$.
\end{defi}

This paper will show that under the condition that $T$ belongs to the class $\mathcal{F}$, the coefficients of an $M_T$-function are always $S$-integers in a number field $\K$. We will also prove that an $M_T$-function is always analytic in a neighborhood of the origin, and that it always extends to a meromorphic function on the open $p$-adic unit ball for all $p \in \mathbf{P} \cup \{ \infty \}$. Finally, combining these results with a theorem of Bell, Chen, Nguyen and Zannier which establishes a rational-transcendental dichotomy for series with $S$-integer coefficients (Corollary 1.3 in \cite{bell2024d}), we will prove a functional rational-transcendental dichotomy result for multivariate $M_T$-functions.

\begin{theo} [Functional Rational-Transcendental dichotomy] \label{thm:dic} $\newline$
Let $T$ be in the class $\mathcal{F}$. Let $f$ be an $M_T$-function. Then either $f$ is in $\A(\z)$, or $f$ is transcendental over $\A(\z)$.
\end{theo}

The regularity property for extensions generated by families of $M_T$-functions follows from this result, and therefore this will complete the proof of \Cref{thm:lift}.\\

\noindent\textit{Acknowledgement.}
The present work is part of the author's PhD thesis under the supervision of Boris Adamczewski. The author would like to thank him for his guidance and his many valuable pieces of advice throughout this work.  The author is also grateful to Benjamin Schraen for discussions on $p$-adic complex analysis, and to Colin Faverjon for his suggestions that simplified the proof of the analyticity of Mahler functions.

\section{Definitions and basic properties of $M_T$-functions}
\subsection{Basic definitions} $\newline$

In this subsection, we give the basic definitions and properties of multivariate Mahler functions. These are very similar in spirit to the univariate Mahler function theory, but they had not been formally stated in this general setting. 

\begin{nt} \label{not:1} We use the convention $\N=\{0,1,2,\ldots\}$ and $\N^*=\{1,2,\ldots\}$.\\
In this paper, $\z$ will denote the $n$-tuple of variables $(z_1,\ldots,z_n)$. We will also use bold letters for other $n$-tuples, like $\al$ for an algebraic point $(\alpha_1,\ldots,\alpha_n)$ or $\bm{k}$ for a vector of integers $(k_1,\ldots,k_n)$. For $\bm{k}$ and $\bm{j}$ two $n$-tuples of integers, we define their sum naturally as the $n$-tuple $(k_1+j_1,\ldots,k_n+j_n)$, and similarly for their difference.\\
For $\z$ and $\bm{k}$ as above, we define the monomial $\z^{\bm{k}}=z_1^{k_1}\ldots z_n^{k_n}$. For a vector of integers $\bm{k}=(k_1,\ldots,k_n) \in \N^n$, we define its length as $|\bm{k}|=k_1+\ldots+k_n$. We also define the length of a monomial $\z^{\bm{k}}$ as $\len(\z^{\bm{k}})=|\bm{k}|$ (its total degree).\\
If $T$ is a matrix with nonnegative integer coefficients, we define $$T\z=(z_1^{t_{11}}z_2^{t_{21}} \ldots z_n^{t_{n1}}, \ldots ,z_1^{t_{1n}}z_2^{t_{2n}} \ldots z_n^{t_{nn}}) \, .$$
This is known as a monomial map (or transformation) of $\C^n$. To distinguish this transformation from the classical linear transformation associated to $T$, we will denote by $T(\bm{x})$ the result of the linear action of $T$ on a (column) vector $\bm{x} \in \C^n$. We note the formula $T\z^{\bm{k}}=\z^{T(\bm{k})}$.\\
We note that there are two notational conventions to define the matrix $T$ associated to a monomial map. In this paper, we choose to use the column convention, where column $i$ describes the image of $z_i$ under the monomial map. If instead we used the row convention, then we would have to replace $T$ by $T^*$ in all the results. With our convention, we note that
$$T\cdot T' z_i=(T \times T')z_i \ \text{but} \ T \cdot  T'\z=(T' \times T) \z \, .$$
\end{nt}
We consider a monomial map defined by a matrix $T \in \M_n(\N)$ as above. We refer to \Cref{defi:admi} and \Cref{defi:f} for the definition of the two main classes of monomial maps: the class $\mathcal{T}$ and the class $\mathcal{F}$.

\begin{rk}
If $T$ is in the class $\mathcal{T}$, then $T$ is in the class $\mathcal{F}$. Indeed, if $T^kz_i=z_i$, then $1$ is an eigenvalue of $T^k$, and therefore (at least) one of the eigenvalues of $T$ is a root of unity.
Furthermore, the converse does not hold, as shown by the example of the matrix
$$\begin{pmatrix}
    1& 0 \\
    2 & 3
\end{pmatrix} \, .$$
\end{rk}
The general definition of a Mahler function is the following.

\begin{defi}
We say that $f(z_1,\ldots,z_n)=f(\z)$ is a $T$-Mahler function if it satisfies an equation of the form
$$\sum_{i=0}^m{p_i(\z)f(T^i\z)}=0$$
where the $p_i$ are polynomials with algebraic coefficients of the $n$ variables $z_1,\ldots,z_n$ and the $p_i$ are not all zero. If $p_m \neq 0$, $m$ is the order of the equation.\\
By convention, when $T=\emptyset$ denotes the empty matrix (with $n=0$), a $T$-Mahler function is defined to be a constant function ($f(\z)=\alpha$ with $\alpha \in \A$).
\end{defi}

In this definition, $f$ is understood as an element of a difference ring extension of $(\A(\z),T)$. For instance, $f$ can be a formal power series but also a formal Laurent series, a Puiseux series or a Hahn series. In the specific case of formal power series with algebraic coefficients, we call these $M_T$-functions, which are the main objects of study in this paper.

\begin{defi}
An $M_T$-function is a formal power series $f \in \A[[\z]]$ which is a $T$-Mahler function.
\end{defi}

Throughout the paper, we will always assume that $T$ is in the class $\mathcal{F}$.

There is a standard characterization of Mahler functions using the vector space generated by applying the Mahler operator to $f$ successively.
\begin{prop} \label{prop:1}
The following are equivalent:
\begin{enumerate}
    \item $f$ is a $T$-Mahler function.
    \item The space $V_f=\Span_{\A(\z)}(T^kf, k \in \N)$ is finite dimensional over $\A(\z)$.
\end{enumerate}
Furthermore, if these are verified, the dimension of $V_f$ is the minimal order for a $T$-Mahler equation satisfied by $f$.
\end{prop}

\begin{proof} 
First, we assume that $f$ is a $T$-Mahler function. We let 
$$\sum_{i=0}^m{p_i(\z)f(T^i\z)}=0$$
be a Mahler equation satisfied by $f$ with $p_m \neq 0$. Applying successive powers of $T$ to this relation, we obtain that $T^pf \in \Span_{\A(\z)}(T^kf, k \le p-1)$ for all $p\ge m$. By induction, we obtain $V_f \subseteq \Span_{\A(\z)}(T^kf, k \le m-1)$, therefore $V_f$ is finite dimensional and $\dim(V_f) \le m$.\\

Conversely, assume that $V_f$ has finite dimension $m$. Then $(f,Tf,\ldots,T^mf)$ must be linearly dependent over $\A(\z)$. This means that $f$ satisfies a $T$-Mahler equation of order $\le m$.\\

The last part follows from the proof of the first two points.
\end{proof}

Another elementary fact that we will need is that the rational functions are $M_T$-functions for all $T$. This follows from \Cref{prop:inhom} below about inhomogeneous $T$-Mahler equations. We first need an easy lemma.

\begin{lem} \label{lem:1}
If $p \in \A[\z]$, $p(T\z)=0$ if and only if $p=0$.
\end{lem}

\begin{proof}
$T$ is non-singular so it is in $\GL_n(\Q)$. The \enquote{Puiseux monomial map} \\ $\z \rightarrow T^{-1}\z$ associated to the matrix $T^{-1}$ gives a left inverse of $T$, hence $T$ is injective.
\end{proof}

\begin{prop} \label{prop:inhom}
If $f$ satisfies an inhomogeneous $T$-Mahler equation of order $\leq m$
$$\sum_{i=0}^m{p_i(\z)f(T^i\z)}=p_{-1}(\z)$$
(with the $p_i$ for $i \in \llbracket     -1,m \rrbracket$ not all zero), then $f$ satisfies a homogeneous $T$-Mahler equation of order at most $m+1$.\\
In particular, rational functions are solutions of a $T$-Mahler equation of order at most $1$.
\end{prop}
\begin{proof} 
We may assume $p_{-1} \neq 0$, and we may also assume $p_m \neq 0$ up to choosing a smaller $m$. We evaluate the inhomogeneous equation at $T\z$ to obtain a second equation. Then, we multiply the two equations by appropriate polynomials and subtract them. We obtain
$$p_m(T\z)p_{-1}(\z)f(T^{m+1}\z)+\sum_{i=1}^m{q_i(\z)f(T^i\z)}-p_{-1}(T\z)p_0(\z)f(\z)=0$$
with $q_i(\z)=p_{-1}(\z)p_{i-1}(T\z)-p_{-1}(T\z)p_i(\z)$. We know that $p_m(T\z)p_{-1}(\z) \neq 0$ by \Cref{lem:1}. This is a $T$-Mahler equation for $f$ of order at most $m+1$.
\end{proof}

Next we prove that we can always find a Mahler equation with $p_0 \neq 0$.

\begin{prop} \label{thm:1}
Let $f$ be an $M_T$-function. Then $f$ satisfies a $T$-Mahler equation with $p_0 \neq 0$.
\end{prop}

\begin{proof}
The function $f$ satisfies an equation of the form
$$\sum_{l=0}^m{p_l(\z)f(T^l\z)}=0$$
with the $p_i$ not all zero. We can assume the order $m$ to be minimal. Assume for a contradiction that $p_0=0$. With a change of index, we obtain
$$\sum_{l=0}^{m-1}{p_{l+1}(\z) \cdot f(T^{l+1}\z)}=0 \, .$$
Consider the monomial transformation given by $T^{-1} \in M_n(\Q)$. This naturally defines a map
$$T^{-1}: p(\z) \in \A[\z] \longrightarrow p(T^{-1}\z) \in \A(z_1^{\frac{1}{d}},\ldots,z_n^{\frac1d})$$
where $d=|\det(T)| >0$.\\
This monomial transformation is also well defined as a map
$$T^{-1} : T(\A[[\z]]) \longrightarrow \A[[\z]]$$
and it is the left inverse of $T$.\\
Applying this transformation $T^{-1}$ to the last equality, we get
$$\sum_{l=0}^{m-1}{p_{l+1}(T^{-1}\z) \cdot f(T^{l}\z)}=0$$
with the $p_{l+1}$ not all zero, so the $f(T^l\z)$ for $l \leq m-1$ are linearly dependent over $\A(z_1^{\frac{1}{d}},\ldots,z_n^{\frac1d})$.\\
Furthermore, since the $p_{l+1}(T^{-1}\z)$ are all in
$$\A[z_1^{\frac{1}{d}},\ldots,z_n^{\frac1d},z_1^{-\frac{1}{d}},\ldots,z_n^{-\frac1d}]$$
we can multiply the previous equation by a monomial to get a relation
$$\sum_{l=0}^{m-1}{q_l(\z) \cdot f(T^{l}\z)}=0$$
with the $q_l$ in $\A[z_1^{\frac{1}{d}},\ldots,z_n^{\frac1d}]$.\\
We let $U=\{0,1,\ldots,d-1\}^n$ be the set of minimal representatives of $n$-tuples modulo $d$. Then, for each $l \in \llbracket 0, m-1 \rrbracket$ we can write
$$q_l(\z)=\sum_{\bm{k} \in U}{r_{l,\bm{k}}(\z) \z^{\frac{\bm{k}}{d}}}$$
for some family $r_{l,\bm{k}}(\z)$ of polynomials. We obtain
$$\sum_{\bm{k} \in U}{\left( \sum_{l=0}^{m-1}{r_{l,\bm{k}}(\z) f(T^l\z)} \right)} \z^{\frac{\bm{k}}{d}}=0 \, .$$
Since the
$$ \left( \sum_{l=0}^{m-1}{r_{l,\bm{k}}(\z) f(T^l\z)} \right) \z^{\frac{\bm{k}}{d}}$$
are all Puiseux series with exponents in $\Z^n+\frac{\bm{k}}{d}$ (which are all disjoint sets), we obtain that for all $\bm{k} \in U$,
$$\sum_{l=0}^{m-1}{r_{l,\bm{k}}(\z) f(T^l\z)}=0 \, .$$
These cannot all be trivial relations, else the $p_{l+1}(T^{-1}\z)$ would all be zero (meaning that all the $p_i$ would be zero). 
Thus, $f$ satisfies a (non-trivial) Mahler equation of order at most $m-1$.\\
This is a contradiction, which means that the minimal Mahler equation for $f$ must have $p_0 \neq 0$.
\end{proof}

We can use this last result to convert a scalar equation satisfied by a Mahler function $f$ into system form.

\begin{defi}
A $T$-Mahler system is an equation of the form $$Y(\z)=A(\z) Y(T \z)$$
where $A \in \GL_m(\A(\z))$. The unknown $Y$ is an $m$-dimensional vector with coordinates in a difference ring extension of $\A(\z)$.
\end{defi}
We have the following result which tells us that $T$-Mahler functions are exactly the coordinates of some solution of a $T$-Mahler system.

\begin{prop} \label{prop:comp}
The following results hold.
\begin{enumerate}
    \item Assume that $f$ is a $T$-Mahler function. Let $m$ be the minimal order for a $T$-Mahler equation satisfied by $f$. Then the vector
    $(f, Tf,...,T^{m-1}f)^*$ is solution of a $T$-Mahler system of size $m$.
    \item If $(f_1,\ldots,f_m)^*$ is a solution of a $T$-Mahler system of size $m$, each coordinate $f_i$ is a solution of a $T$-Mahler equation of order at most $m$.
\end{enumerate}
\end{prop}

\begin{proof}
For the first part, if $f$ is a $T$-Mahler function, then the minimal $T$-Mahler equation for $f$ must have $p_0 \neq 0$ using the proof of \Cref{thm:1}. This minimal equation is assumed to have order $m$, thus we get an equation of the form
$$p_0(\z)f(\z)=p_1(\z)f(T\z)+\ldots+p_m(\z)f(T^m\z)$$
for some polynomials $p_i \in \A[\z]$. Then the vector $Y=(f,Tf,\ldots,T^{m-1}f)^*$ satisfies the Mahler system $Y(\z)=A(\z)Y(T\z)$ where
$$A(\z)=\begin{pmatrix}
    \frac{p_1(\z)}{p_0(\z)} & \frac{p_2(\z)}{p_0(\z)} & \dots  & \frac{p_m(\z)}{p_0(\z)} \\
    1 & 0  & \dots  & 0 \\
    
    \vdots & \ddots & \ddots & \vdots \\
    0 & 0 & 1  & 0 \\   
\end{pmatrix} \, .$$
Since $m$ is minimal, we know that $p_m \neq 0$, hence $$\det(A(\z))=(-1)^{m+1} \frac{p_m(\z)}{p_0(\z)} \neq 0 \, .$$
For the second part, if $Y=(f_1,\ldots,f_m)^*$ is a solution of such a system, we get by induction on $k$ that for all $k \in \N$,
$$\forall i \in \llbracket 1,m \rrbracket, f_i(T^k\z) \in  \Span_{\A(\z)}(f_1,\ldots ,f_m)$$
so for all $i \in \llbracket 1,m \rrbracket$ , $V_{f_i} \subseteq \Span_{\A(\z)}(f_1,\ldots ,f_m)$, and therefore $$\dim{V_{f_i}} \le m \, .$$
\end{proof}

\subsection{The difference ring $\mathfrak{M}_T$} $\newline$

In this subsection, we prove basic stability results for $M_T$-functions.

\begin{prop} \label{prop:stab} The following properties hold:
\begin{enumerate}
    \item If $f$ and $g$ are $T$-Mahler then $f+g$ and $fg$ are $T$-Mahler. In particular, $\mathfrak{M}_T$ is a ring, and even a $\A[\z]$-algebra.
    \item $f$ is $T$-Mahler if and only if $Tf$ is $T$-Mahler.
    \item For all $k \in \N^*$, $f$ is $T$-Mahler if and only if $f$ is $T^k$-Mahler.
    \item If $f$ is $T$-Mahler, then $\partial_{z_i}f$ is also $T$-Mahler for all $i \in  \llbracket     1,n \rrbracket$.
\end{enumerate}
\end{prop}

\begin{proof}
The proofs of the first three results are very similar to the univariate case, but we still give them for completeness.
\begin{enumerate}
    \item This follows from \Cref{prop:1}, since $V_{f+g} \subseteq V_f+V_g$ and $V_{fg} \subseteq V_f\cdot V_g$ (which are both finite dimensional, with dimension $\le \dim(V_f) + \dim(V_g)$ for $V_{f+g}$ and dimension $\le \dim(V_f) \cdot \dim(V_g)$ for $V_{fg}$).
    \item First, assume that $f$ satisfies a $T$-Mahler equation. Then setting $\z \rightarrow T\z$ in that equation gives a $T$-Mahler equation for $f(T\z)$ (it is not possible that $p_i(T\z)=0$ for all $i$ as that would imply that $p_i=0$ for all $i$ by \Cref{lem:1}).\\
    Conversely, if $f(T\z)$ satisfies a $T$-Mahler equation, it is in particular a $T$-Mahler equation for $f$.
    \item First, assume that $f$ is $T$-Mahler. Since $\Span_{\A(\z)}((T^{k})^if, i \in \N) \subseteq V_f$, it follows that the space $\Span_{\A(\z)}((T^{k})^if, i \in \N)$ has finite dimension over $\A(\z)$ so $f$ is $T^k$-Mahler by \Cref{prop:1}.\\
    Conversely, if $f$ is $T^k$-Mahler, a $T^k$-Mahler equation for $f$ is clearly also a $T$-Mahler equation.
    \item We know that $V_f$ has finite dimension: let $f_1,...,f_m$ be a $\A(\z)$-basis of $V_f$.\\
Let
$$E=\Span_{\A(\z)}(f_1,...,f_m, (\partial_{z_i}f_1,...,\partial_{z_i}f_m, i \in \llbracket     1,n \rrbracket)) \, .$$
By the Leibniz rule (note: $\partial_{z_i}$ isn't $\A(\z)$-linear), we know that if $g \in V_f$, then for all $i \in \vaset$, $\partial_{z_i}g \in E$. Therefore for all $i \in \vaset$ and $k \in \N$, we get $\partial_{z_i}(f(T^k\z)) \in E$.\\
Let $t_{ijk}$ be the $(i,j)$-coefficient of $T^k$. Using the chain rule, we get
$$\partial_{z_i}(f(T^k\z))=\sum_{j=1}^n{t_{ijk}z_i^{t_{ijk}-1}\prod_{l \neq i}{z_l^{t_{ljk}}}(\partial_{z_j}f)
(T^k\z)} \, .$$
Define $M$ as the matrix with $i,j$-coefficient
$$M_{ij}=t_{ijk}z_i^{t_{ijk}-1}\prod_{l \neq i}{z_l^{t_{ljk}}} \, .$$
Then the last equality is just $X=MY$ where $$X=(\partial_{z_i}(f(T^k\z)), i \in \vaset)^*$$ has coefficients in $E$ and $$Y=((\partial_{z_i}f)
(T^k\z), i \in \vaset)^* \, .$$
Since $\det(M)(1,1,...,1)=\det(T^k) \neq 0$, $\det(M)$ is non zero as a polynomial, so $M$ is invertible in $\M_n(\A(\z))$, and thus we can write $Y=M^{-1}X$.\\
Therefore, for all $i \in \vaset$ and $k \in \N$, $(\partial_{z_i}f)
(T^k\z)$ is a $\A(\z)$-linear combination of coefficients of $X$, so it is in $E$. Since $E$ has finite dimension, it follows that $V_{\partial_{z_i}f}$ has finite dimension over $\A(\z)$ for all $i \in \vaset$, which is what we wanted.

\end{enumerate}
\end{proof}

\subsection{Recurrence relation for coefficients of an $M_T$-function} \label{section:2.4} $\newline$

This subsection shows that the coefficients of an $M_T$-function satisfy a recurrence relation.

\begin{nt}
If $(b(k_1,\ldots,k_n))_{(k_1,\ldots,k_n) \in \N^n}$ is a multi-dimensional sequence, and if $(r_1,\ldots,r_n) \in \R^n$, we use the convention $b(r_1,\ldots,r_n)=0$ if $(r_1,\ldots,r_n) \notin \N^n$.
\end{nt}

\begin{defi} For a transformation $T$ and a column vector $X \in \N^n$, we define the set $$E_T(X)=\{Y=(y_1,...,y_n)\in \N^n |T(Y)=X \}$$ where $T(Y)$ is the usual matrix-column vector multiplication.\\
Note that since $T$ is non-singular, $E_T(X)$ is either empty or has only one element.\\
For a sequence $a \in \C^{\N^n}$, define further $$a(T^{-1}(X))=\sum_{(y_1,\ldots,y_n) \in E_T(X)}{a(y_1,\ldots y_n)} \, .$$
This is either $0$ or a single term $a(y_1,...,y_n)$ with $T(y_1,...,y_n)=X$.
\end{defi}

Then the coefficients of an $M_T$-function satisfy a recurrence relation of the following form.

\begin{prop} \label{prop:9}
Let $f(z_1,...,z_n)=\sum_{k_1,\ldots,k_n \geq 0}{a(k_1,...k_n)z_1^{k_1}\ldots z_n^{k_n}}$ be a solution of a $T$-Mahler equation of order $m$. Then there exists $N \in \N$ and constants $c_{j_1, \ldots, j_n, l} \in \A$ such that $a(k_1, \ldots, k_n)$ satisfies the recurrence relation
\begin{align*}
&\sum_{(j_1,\ldots,j_n) \in \llbracket 0,N \rrbracket^n}
c_{j_1,\ldots,j_n,0}\cdot a(k_1-j_1,\ldots,k_n-j_n) \\
&\quad +
\sum_{l=1}^m \sum_{(j_1,\ldots,j_n) \in \llbracket 0,N \rrbracket^n}
c_{j_1,\ldots,j_n,l}\cdot a(T^{-l}(k_1-j_1,\ldots,k_n-j_n))
=0 \, .
\end{align*}
\end{prop}

\begin{proof}
The function $f$ satisfies an equation of the form 
$$\sum_{i=0}^m{p_i(\z)f(T^i\z)}=0 \, .$$
Let $$N=\max_{i,k}\deg_{z_k}{p_i(\z)}$$
and then let $c_{\bm{j},l}$ (for $\bm{j} \in \llbracket     0,N \rrbracket^n$ and $l \in \llbracket     0,m \rrbracket$) be the coefficient of $\z^{\bm{j}}$ in $p_l(\z)$. It is algebraic since the $p_l(\z)$ have algebraic coefficients.\\
Remember that $T\z^{\bm{p}}=\z^{T(\bm{p})}$, so the coefficient in $\z^{\bm{k}}$ of $\z^{\bm{j}}f(T^l\z)$ is $a(T^{-l}(\bm{k}-\bm{j}))$.\\
Writing $f$ as a power series and computing the terms of the $T$-Mahler equation, then looking at the coefficient in $\z^{\bm{k}}$, gives that
$$\sum_{l=0}^m{\sum_{\bm{j} \in \llbracket     0,N \rrbracket^n}{c_{\bm{j}, l} \cdot a(T^{-l}(\bm{k}-\bm{j}))}}=0$$
which is what we wanted.
\end{proof}

\begin{ex} We show in practice what the recurrence relations above look like and how we can use them.\\
Consider the operator
$T=\begin{pmatrix}
5 & 1 \\
4 & 1 
\end{pmatrix}$
and the $T$-Mahler equation
$$(z_1^3z_2^3-2z_1^5z_2^2)f(z_1,z_2)=f(T(z_1,z_2))+(z_1z_2^2)f(T^2(z_1,z_2)) \, .$$
An $M_T$-function solution of the previous equation has the form
$$f(z_1,z_2)=\sum_{n \geq 0, m \geq 0}{a(n,m)z_1^nz_2^m}$$
where $a(n,m)$ satisfies the recurrence relation 
$$a(n-3,m-3)-2a(n-5,m-2)=a(T^{-1}(n,m))+a(T^{-2}(n-1,m-2)) \, .$$
\\
Now, if we make the change of variables $n \rightarrow n+3$, $m \rightarrow m+3$, this recurrence relation becomes
$$a(n,m)=2a(n-2,m+1)+a(T^{-1}(n+3,m+3))+a(T^{-2}(n+2,m+1)) \, .$$
We can see that $|(n-2,m+1)| <|(n,m)|$. Furthermore, $$|T(n,m)|=|(5n+m,4n+m)| \geq 2n+2m=2|(n,m)|$$
which gives
$$|T^{-1}(n,m)| \leq \frac12 |(n,m)| \, .$$
Hence, we see that
$$|T^{-1}(n+3,m+3)| \leq \frac12 |(n+3,m+3)|=\frac{|(n,m)|+6}{2} <|(n,m)|$$
for $n+m >6$. Similarly $|T^{-2}(n+2,m+1)|<|(n,m)|$ for $n+m >1$.

This shows that for $n+m >6$, this recurrence relation expresses $a(n,m)$ as a linear combination of terms $a(p,q)$ where $|(p,q)| < |(n,m)|$. This means that the sequence $a(n,m)$ is completely determined by its initial terms (for which $n+m \leq 6$) as well as the recurrence relation. This will be made more general in \Cref{section:3} to obtain the main properties of coefficients of $M_T$-functions.
\end{ex}

\subsection{Preliminary reduction of $T$} \label{section:red} $\newline$

In this section, we perform a reduction on $T$ so that we can get suitable bounds on the terms $(p_1,\ldots,p_n)$ such that $a(p_1,\ldots,p_n)$ appears on the right side of the recurrence relation.

\begin{defi} \label{defi:order} We write $\leq_{\rm{lex}}$ for the lexicographic order on $\N^n$. We then define a new order $\leq$ on $\N^n$ by $\bm{k}\leq\bm{k'}$ if and only if
$$|\bm{k}| < |\bm{k'}| \ \text{or} \ (|\bm{k}| =|\bm{k'}| \ \text{and} \ \bm{k} \leq_{\rm{lex}} \bm{k'}) \, .$$
This gives a total order on $\N^n$, which extends to a total order on monomials.
\end{defi}

We recall that \Cref{prop:stab} gives that the set of $M_T$-functions and the set of $M_{T^k}$-Mahler functions are the same. This motivates the following definition.

\begin{defi}
We say that $T$ reduces to $T'$ if there exists $k \in \N^*$ such that $T'=T^k$. Note that if $T$ belongs to the class $\mathcal{F}$, then $T'$ will also belong to the class $\mathcal{F}$.
\end{defi}

The main result that we will use in \Cref{section:3} is the following.

\begin{prop} \label{prop:red}
Let $T$ be in the class $\mathcal{F}$. Then $T$ reduces to a matrix $T'$ such that
\begin{enumerate}
    \item For all $i \in \vaset$, $[T']_{(i,i)} \geq 1$.
    \item For all $\bm{k} \in \N^n$, $|T'(\bm{k})| \geq 2 |\bm{k}|$, hence 
    $$|(T')^{-1}(\bm{k})| \leq \frac12 |\bm{k}|$$ for all $\bm{k}$ such that $(T')^{-1}(\bm{k}) \in \N^n$.
\end{enumerate}
    
\end{prop}

Note that these conditions are stable under taking powers again. We separate the proof of \Cref{prop:red} into two lemmas, one for each condition.

\begin{lem}
If $T$ is non-singular, then $T$ reduces to a matrix $T'$ such that for all $i \in \vaset$, $[T']_{(i,i)} \geq 1$.
\end{lem}

\begin{proof}
Note that if $T^k$ has a non zero coefficient in the spot $(i,i)$, then all powers of $T^k$ will still have a non zero coefficient in the spot $(i,i)$. Therefore it is sufficient to prove that for all $i \in \vaset$, there exists $k$ such that $T^k$ has a non zero coefficient in the spot $(i,i)$. We prove that for $i=1$, the proof being identical for the other values of $i$.\\
Consider the set $S \subseteq \vaset$ of variables appearing in the chain
$$Tz_1 \rightarrow T^2z_1 \rightarrow\ldots$$
and let $d=|S \cup \{1\}|$. The goal is equivalent to showing that $1 \in S$.\\
Then consider the submatrix $X \in M_{n \times d}(\N)$ given by the $d$ columns of $T$ corresponding to elements of $S \cup \{1\}$. Since $T$ is non-singular, $X$ is of rank $d$. By definition, for $i \in S \cup\{1\}$, $Tz_i$ can only contain variables $z_j$ with $j \in S$. Therefore, all rows of $X$ corresponding to $j \notin S$ are zero, and so the rank of $X$ is $\leq |S|$. This gives $d=|S \cup \{1\}| \leq |S|$ and so $1 \in S$.
\end{proof}

\begin{lem} \label{lem:expending}
Let $T$ be in the class $\mathcal{F}$. Then
\begin{enumerate}
    \item For all $\bm{k} \in \N^n$, we have $|T^n(\bm{k})| \geq 2 |\bm{k}|$.
    \item The length of $T^k z_i$ tends to $+\infty$ for all $i \in \vaset$.
\end{enumerate}
\end{lem}

\begin{proof}
By linearity, we only need to prove that $|T^n(e_i)| \geq 2$ for $$e_i=(0,0,\ldots,1,0\ldots,0)^*$$ the $i$-th vector of the canonical basis. Equivalently, we need to prove that the length of $T^nz_i$ is at least $2$ for all $i \in \vaset$.\\
Since $T$ is non-singular, its columns cannot be zero, so $\len(Tz_i) \geq 1$ for all $i \in \vaset$. Therefore, $\len(T^k z_i)$ is non-decreasing (as a function of $k$). Assume that there is some $i \in \vaset$ such that $\len(T^nz_i)=1$. Then for all $k \in \llbracket 0, n \rrbracket$, $T^kz_i=z_{a_k}$ for some $a_k \in \vaset$. Since the $a_k$ cannot be all distinct, there exists $k_1<k_2$ such that $T^{k_1}z_i=T^{k_2}z_i$. But then $T^{k_2-k_1}z_{a_{k_1}}=z_{a_{k_1}}$, which contradicts the fact that $T$ is in the class $\mathcal{F}$. We conclude that $\len(T^n z_i) \geq 2$ for all $i \in \vaset$, which proves the first part of the lemma.\\

For the second part, we iterate the first result to get that $\len(T^{kn}z_i) \geq 2^k$ for all $k \in \N$, and therefore $\len(T^k z_i)$ is non-decreasing and unbounded, so it tends to $+\infty$.
\end{proof}

\section{Main functional properties of $M_T$-functions} \label{section:3}

In this section, we prove the main functional properties of $M_T$-functions, namely that the coefficients of an $M_T$-function are $S$-integers in a common number field $\K$, and that an $M_T$-function is analytic in a neighborhood of $\bm{0}$ as well as meromorphic on the open polydisk of radius $1$ in $\C^n$. Furthermore, we prove that these analytic properties still hold in the $p$-adic setting.

\subsection{Arithmetic properties of coefficients of $M_T$-functions} \label{sect:3.1} $\newline$

For a number field $\K$ of degree $d$, we let $\mathcal{M}(\K)$ be the set of places of $\K$ with $\mathcal{M}_{\infty}(\K)$ denoting the (finite) set of infinite places of $\K$. For a place $\nu \in \mathcal{M}(\K)$, we write $|\cdot|_{\nu}$ for the corresponding absolute values.\\
The infinite places correspond to the archimedean absolute values given by $$|\cdot|_{\infty,\sigma}=|\sigma(\cdot)|^{i}$$
for each embedding $\sigma$ of $\K$ into $\C$. We normalize these infinite places by taking $i=\frac{1}{d}$ if $\sigma$ is a real embedding and $i=\frac{2}{d}$ if $\sigma$ is a complex embedding.\\
For each prime $p$, we fix a complete algebraically closed field
$(\C_p,|\cdot|_p)$ containing the field of $p$-adic numbers $\Q_p$ in the sense that the absolute value $|\cdot|_p$ on $\C_p$ restricts to an absolute value on $\Q_p$ which is equivalent to the usual $p$-adic absolute value.\\
Then the finite places correspond to the non-archimedean absolute values $$|\cdot|_{p,\sigma}=|\sigma(\cdot)|^i_p$$
for each embedding $\sigma$ of $\K$ into $\C_p$. If $\nu$ is an absolute value as above, we let $d_v$ be the quantity $[\K_v,\Q_p]$ where $\K_v$ is the completion of $\K$ using the place $\nu$. We normalize these non-archimedean absolute values by taking $i=\frac{d_v}{d}$.
These normalizations are used to obtain the product formula
$$\forall x \in \K^*, \prod_{\nu \in \mathcal{M}(\K)}{|x|_{\nu}}=1$$
and the expression of the logarithmic Weil height
$$\forall x \in \K^*, h(x)=\sum_{\nu \in \mathcal{M}(\K)}{\max(0,\log|x|_{\nu})}$$
with the convention $h(0)=0$.
We refer to Chapter 3 of \cite{waldschmidt2000diophantine} for an overview of these notions.\\
Recall the definition of an $S$-integer.
\begin{defi} 
Let $S$ be a finite set of places of $\K$
containing the archimedean ones, and let $x \in \K$. We say that $x$ is an $S$-integer and we write $x \in \mathcal{O}_{\K,S}$ if for all $\nu \notin S$, we have $|x|_{\nu} \leq 1$.\\
This is a ring because the finite places correspond to non-archimedean absolute values satisfying the ultrametric inequality.
\end{defi}

The main result of this subsection is that the coefficients of an $M_T$-function are all $S$-integers in a number field $\K$. This uses the results in \Cref{section:2.4} as well as \Cref{section:red}.

\begin{prop} \label{prop:field}
Let  $f$ be an $M_T$-function with $T$ in the class $\mathcal{F}$.\\
Then there exists a number field $\K$ containing all the coefficients of $f$.\\
Furthermore, there exists $S$ a finite set of places of $\K$ (containing the archimedean ones) such that all coefficients of $f$ are $S$-integers.
\end{prop}

\begin{proof} 
By \Cref{prop:red}, we can reduce $T$ to a matrix $T'$ satisfying
\begin{enumerate}
    \item For all $i \in \vaset$, $[T']_{(i,i)} \geq 1$.
    \item For all $\bm{k} \in \N^n$, $|T'(\bm{k})| \geq 2 |\bm{k}|$, hence 
    $$|(T')^{-1}(\bm{k})| \leq \frac12 |\bm{k}|$$ for all $\bm{k}$ such that $(T')^{-1}(\bm{k}) \in \N^n$.
\end{enumerate}
Therefore, we will directly assume that $T$ satisfies these conditions.\\
We know that $f$ satisfies a $T$-Mahler equation
$$\sum_{l=0}^m{p_l(\z)f(T^l\z)}=0$$
with $p_0 \neq 0$ by using \Cref{thm:1}.\\
Now by \Cref{prop:9}, we get that $a(k_1,...,k_n)$ satisfies a recurrence relation of the form
\begin{align*}
\sum_{\bm{j} \in \llbracket 0,N \rrbracket^n}
c_{\bm{j},0} \cdot a(\bm{k}-\bm{j})
+
\sum_{l=1}^m \sum_{\bm{j} \in \llbracket 0,N \rrbracket^n}
c_{\bm{j},l} \cdot a(T^{-l}(\bm{k}-\bm{j}))
=0
\end{align*}
with the $c_{\bm{j}, 0}$ (the coefficients of $p_0$) which are not all zero.\\
Let $\bm{\nu}=(\nu_1,...,\nu_n)$ be the minimal $\bm{j}=(j_1,...,j_n)$ for the order defined in \Cref{defi:order} such that $c_{\bm{j}, 0} \neq 0$. Then write the recurrence relation for $\bm{k} \rightarrow \bm{k}+\bm{\nu}$. After isolating $a(\bm{k})$, this gives
\begin{align*}
   a(\bm{k})&=\sum_{\substack{\bm{j} \in \llbracket     0,N \rrbracket^n \\ \bm{j} \neq \bm{\nu}}}{\frac{-c_{\bm{j}, 0}}{c_{\bm{\nu}, 0}} \cdot a(\bm{k}+\bm{\nu}-\bm{j})}
    -\sum_{l=1}^m{\sum_{\bm{j} \in \llbracket     0,N \rrbracket^n}{\frac{c_{\bm{j}, l}}{c_{\bm{\nu}, 0}} \cdot a(T^{-l}(\bm{k}+\bm{\nu}-\bm{j}))}} \, .&
\end{align*}
Now in this equation, all non-zero terms with $a(\bm{k}+\bm{\nu}-\bm{j})$ appearing on the right-hand side have $\bm{k}+\bm{\nu}-\bm{j} < \bm{k}$ (for the order of \Cref{defi:order}) by construction of $\bm{\nu}$.\\
For all $1 \leq l \leq m$, and for all $\bm{j} \in \llbracket     0,N \rrbracket^n$, we observe that if $$\bm{p} \in E_{T^l}(\bm{k}+\bm{\nu}-\bm{j})$$ then we have $T^l(\bm{p})=\bm{k}+\bm{\nu}-\bm{j}$ and so
$$|\bm{p}|\leq \frac{1}{2^l} |\bm{k}+\bm{\nu}-\bm{j}| \leq \frac{|\bm{k}|+|\bm{\nu}|}{2}$$
using the second condition of \Cref{prop:red}.\\
This implies that if $|\bm{k}| > |\bm{\nu}|$, then for all $\bm{p} \in E_{T^l}(\bm{k}+\bm{\nu}-\bm{j})$, we have $|\bm{p}| < |\bm{k}|$.\\

Therefore, for $|\bm{k}| > |\bm{\nu}|$, the recurrence relation expresses $a(\bm{k})$ as a $\A$-linear combination of terms $a(\bm{p})$ where $\bm{p}< \bm{k}$. \\

Let $\K$ be a number field containing all the coefficients $c_{j_1, \ldots, j_n, l}$ and the (finite number of) coefficients $a(\bm{k})$ where $|\bm{k}| \leq |\bm{\nu}|$. We claim that all coefficients $a(\bm{k})$ of $f$ are in $\K$.\\
We prove this claim by induction on $\bm{k} \in \N^n$ with the order $\leq$. The initialization is satisfied for $\bm{k}=(0,...,0)$.\\
For the induction step, if $|\bm{k}| \leq |\bm{\nu}|$, then $a(\bm{k})$ is in $\K$ by definition. If $|\bm{k}| > |\bm{\nu}|$, then $a(\bm{k})$ is a $\K$-linear (because $\K$ contains the $c_{\bm{j}, l}$ and $\frac{1}{c_{\bm{\nu},0}}$) combination of terms $a(\bm{p})$ with $\bm{p}< \bm{k}$. These $a(\bm{p})$ are in $\K$ by induction hypothesis. We conclude that in that case, $a(\bm{k})$ is in $\K$ as well.\\
Therefore, all coefficients of $f$ are in $\K$.\\

For the second part about $S$-integers, we know that $\K$ contains the coefficients of the recurrence relation $c_{\bm{j}, l}$ and the (finite number of) coefficients $a(\bm{k})$ where $|\bm{k}| \leq |\bm{\nu}|$.\\
For a fixed $x \in \K$, there is a finite set of places of $\K$ that we will denote by $S_x$, such that $x$ is an $S_x$-integer (we can assume that $S_x$ contains the archimedean places).
Taking $S$ to be the union of the $S_x$ with $x$ ranging over all $\frac{c_{\bm{j}, l}}{c_{\bm{\nu}, 0}}$, as well as the initial coefficients $a(\bm{k})$ for $|\bm{k}| \leq |\bm{\nu}|$,  we see that all coefficients of the recurrence relation as well as the initial conditions are $S$-integers. As above, an easy induction shows that all coefficients of $f$ are $S$-integers (only using that the set of $S$-integers is a ring).\\
This ends the proof of the proposition.
\end{proof}

\subsection{Analyticity of $M_T$-functions at 0} \label{sect:3.2} $\newline$

An $M_T$-function is a priori only a formal power series. However, we can show that the recurrence relation satisfied by the coefficients automatically implies that it is a convergent power series (in a small neighborhood of $\bm{0}$). For this we use the same strategy as in the last subsection.

\begin{prop} [Analyticity at $\bm{0}$] \label{thm:5} $\newline$
Let $f(z_1,...,z_n)=\sum_{k_1,\ldots,k_n \geq 0}{a(k_1,...k_n)z_1^{k_1}\ldots z_n^{k_n}} \in \A[[z_1,...,z_n]]$ be an $M_T$-function with $T$ in the class $\mathcal{F}$.\\
Let $\K$ be a number field containing all of the coefficients of $f$ (and the coefficients of the recurrence relation for the coefficients of $f$). Then there is a $C>0$ such that for all $\nu \in \mathcal{M}(\K)$, we have
$$\forall (k_1,\ldots,k_n) \in \N^n, |a(k_1,\ldots,k_n)|_{\nu} \le C^{k_1+\ldots+k_n+1} \, .$$
In particular, $f$ is analytic in a neighborhood of $\bm{0}$ in $\C^n$.
\end{prop}

The strategy to prove \Cref{thm:5} will be very similar to what we did in the last section to prove \Cref{prop:field}. We start by proving two lemmas.

\begin{lem} \label{lem:6} 
Let $N \in  \N^*$. Assume that $(p_1,...,p_n)<_{\rm{lex}}(k_1,...,k_n)$ and $p_j \leq k_j+N$ for all $j$. Then for $M=N+1$,
$$\sum_{j=1}^n{p_jM^{n-j}} \leq \sum_{j=1}^n{k_jM^{n-j}}-1 \, .$$
\end{lem}

\begin{proof}[Proof of \Cref{lem:6}] $\newline$
Let $j_0$ be the smallest $j$ such that $p_j \neq k_j$. Then by definition, $p_{j_0} < k_{j_0}$ and $p_j=k_j$ for all $j <j_0$.\\
Then $p_{j_0} \leq k_{j_0}-1$ so
\begin{align*}
 \sum_{j=1}^n{p_jM^{n-j}} &\leq \sum_{j=1}^{j_0-1}{k_jM^{n-j}}+k_{j_0}M^{n-j_0}-M^{n-j_0}+\sum_{j=j_0+1}^n{(k_j+N)M^{n-j}}&\\
 &=\sum_{j=1}^n{k_jM^{n-j}}+N\sum_{j=j_0+1}^n{M^{n-j}}-M^{n-j_0} \, .&
\end{align*}
Since $$N\sum_{j=j_0+1}^n{M^{n-j}}=N\sum_{j=0}^{n-j_0-1}{M^j}=N\frac{M^{n-j_0}-1}{M-1} = M^{n-j_0}-1$$
we get what we wanted.
\end{proof}

\begin{lem} \label{lem:weights} Let $N \in \N^*$. Assume that $(p_1,...,p_n)<(k_1,...,k_n)$ for the order of \Cref{defi:order}, and $p_j \leq k_j+N$ for all $j$. Then for $M=N+1$,
$$M^n |\bm{p}|+\sum_{j=1}^n{p_jM^{n-j}} \leq M^n |\bm{k}|+\sum_{j=1}^n{k_jM^{n-j}}-1 \, .$$
\end{lem}

\begin{proof}[Proof of \Cref{lem:weights}] $\newline$
By definition of the order, either $|\bm{p}|<|\bm{k}|$ or ($|\bm{p}|=|\bm{k}|$ and $\bm{p} <_{\rm{lex}} \bm{k}$). In the second case, the result follows from the previous lemma.\\
In the first case, we know that
\begin{align*}
M^n |\bm{p}|+\sum_{j=1}^n{p_jM^{n-j}}
&\leq M^n (|\bm{k}|-1)+\sum_{j=1}^n{(k_j+N)M^{n-j}}&\\
&\leq M^n |\bm{k}|+\sum_{j=1}^n{k_jM^{n-j}}-M^n+N\sum_{j=1}^n{M^{n-j}}&\\
&=M^n |\bm{k}|+\sum_{j=1}^n{k_jM^{n-j}}-1
\end{align*}
which finishes the proof.
\end{proof}

We note that this lemma shows that we can mimic our order on $\N^n$ by using the weights 
$$\omega_M(\bm{k})=M^n |\bm{k}|+\sum_{j=1}^n{k_jM^{n-j}}$$
as long as we restrict to tuples for which the coefficients do not grow too much.\\
We can now prove $\Cref{thm:5}$.

\begin{proof} [Proof of \Cref{thm:5}] $\newline$
As in the proof of \Cref{prop:field}, we can assume that $T$ satisfies the properties
\begin{enumerate}
    \item For all $i \in \vaset$, $[T]_{(i,i)} \geq 1$.
    \item For all $\bm{k} \in \N^n$, $|T(\bm{k})| \geq 2 |\bm{k}|$, hence 
    $$|T^{-1}(\bm{k})| \leq \frac12 |\bm{k}|$$ for all $\bm{k}$ such that $T^{-1}(\bm{k}) \in \N^n$.
\end{enumerate}

Proceeding the same way as in the proof of \Cref{prop:field}, we again get that $a(\bm{k})$ satisfies a recurrence relation of the form
\begin{align*}
   a(\bm{k})&=\sum_{\substack{\bm{j} \in \llbracket     0,N \rrbracket^n \\ \bm{j} \neq \bm{\nu}}}{\frac{-c_{\bm{j}, 0}}{c_{\bm{\nu}, 0}} \cdot a(\bm{k}+\bm{\nu}-\bm{j})}
    -\sum_{l=1}^m{\sum_{\bm{j} \in \llbracket     0,N \rrbracket^n}{\frac{c_{\bm{j}, l}}{c_{\bm{\nu}, 0}} \cdot a(T^{-l}(\bm{k}+\bm{\nu}-\bm{j}))}}&
\end{align*}
where the $c_{\bm{j},l}$ are the coefficients of the polynomials $p_l$ appearing in a $T$-Mahler equation for $f$ (with $p_0 \neq 0$), $N$ is the maximal degree in any variable of the $p_l$, and $\bm{\nu}$ is the valuation of $p_0$ for the order of \Cref{defi:order}.\\
As before, for $|\bm{k}|>|\bm{\nu}|$, all terms $a(\bm{p})$ appearing on the right side of this equation have $\bm{p}<\bm{k}$.\\
Furthermore, all terms $a(\bm{p})$ appearing also have $p_m \leq k_m+N$ for all $m \in \vaset$.\\
Indeed, if $\bm{p} \in E_{T^l}(\bm{k}+\bm{\nu}-\bm{j})$, then
$$t_{m1l}p_1+\ldots+t_{mml}p_m+\ldots+t_{mnl}p_n=k_m+\nu_m-j_m$$
which gives
$$p_m \leq k_m+\nu_m-j_m \leq k_m +N$$
since $t_{mml}\geq 1$ using the first condition in \Cref{prop:red} (note that $T^l$ still satisfies the conditions).
\\
We let $M=N+1$. We let $C>1$ be large enough so that
$$|a(\bm{k})| \leq C^{M^n |\bm{k}|+M^{n-1}k_1+\ldots+Mk_{n-1}+k_n+1}=C^{\omega_M(\bm{k})+1}$$
for all $\bm{k}$ such that $|\bm{k}| \leq |\bm{\nu}|$ (this is a finite number of conditions), and we also ask that $C>L(m+1)(N+1)^n$ where $L=\max\{\left | \frac{c_{\bm{j}, l}}{c_{\bm{\nu}, 0}} \right|  \}$.\\
We now prove by induction on $\bm{k}$ that
$$|a(\bm{k})| \leq C^{\omega_M(\bm{k})+1}$$
for all $\bm{k} \in \N^n$.

We know that the inequality is true for $\bm{k}=(0,...,0)$ by construction of $C$.\\
Assume that the inequality is true for all $\bm{p} <\bm{k}$.\\
If $|\bm{k}|\leq |\bm{\nu}|$ the inequality holds by construction.\\
If $|\bm{k}|> |\bm{\nu}|$, we use the recurrence relation. We know that for all terms $a(\bm{p})$ appearing on the right side of the equation, we have $\bm{p}<\bm{k}$. By \Cref{lem:weights}, we obtain that for all terms $a(\bm{p})$ appearing on the right side of the recurrence relation, we have $\omega_M(\bm{p}) \leq \omega_M(\bm{k})-1$, and therefore

$$|a(\bm{p})| \leq C^{\omega_M(\bm{p})+1}
\leq C^{\omega_M(\bm{k})}
=\frac{1}{C}(C^{\omega_M(\bm{k})+1})
$$
(using the induction hypothesis).\\
Since there are at most $(m+1)(N+1)^n$ such terms on the right side of the recurrence relation, we obtain
$$|a(\bm{k})| \leq L(m+1)(N+1)^n\cdot \frac{1}{C}(C^{\omega_M(\bm{k})+1}) \leq C^{\omega_M(\bm{k})+1} \, .$$
This ends the proof of the inequality 
$$|a(k_1,...,k_n)| \leq C^{M^n |\bm{k}|+M^{n-1}k_1+\ldots+Mk_{n-1}+k_n+1}\, .$$
Since $$M^n |\bm{k}|+M^{n-1}k_1+\ldots+Mk_{n-1}+k_n \leq (M^n+M^{n-1})|\bm{k}|$$
we obtain
$$|a(k_1,\ldots,k_n)| \leq (C')^{k_1+\ldots+k_n+1}$$
with $C'=C^{M^n+M^{n-1}}$.\\
This shows that $f$ is analytic on the open ball of radius $r=\frac{1}{C'} >0$.\\

Now, for any $\nu \in \mathcal{M}(\K)$,  doing the exact same thing but replacing the standard absolute value by the absolute value $|\cdot|_{\nu}$, we also get an inequality of the form
$$|a(k_1,\ldots,k_n)|_{\nu} \leq C_{\nu}^{k_1+\ldots+k_n+1}$$
for some $C_{\nu} >1$. Let $S$ be a finite set of places of $\K$ containing the archimedean ones such that all coefficients of $f$ are $S$-integers. We take $C=\max_{\nu \in S}{C_{\nu}}>1$ (which is well defined as $S$ is finite). Then the inequality
$$|a(k_1,\ldots,k_n)|_{\nu} \leq C^{k_1+\ldots+k_n+1}$$
holds automatically for $\nu \in S$, but it also holds for $\nu \notin S$, since in that case $|a(k_1,\ldots,k_n)|_{\nu}\leq 1$.
\end{proof}

\subsection{Meromorphy on the unit ball} \label{sect:3.3} $\newline$

The main result of this subsection is that any $M_T$-function can be extended to a meromorphic function on the ball $B(0,1)$. In fact there is a natural way to define a domain $U_T$ (which in general strictly contains $B(0,1)$) so that any $M_T$-function extends as a meromorphic function on $U_T$. Let us start with the definition of this domain $U_T$.

\begin{defi}
If $\z \in \C^n$, we define $||\z||=\max\{|z_i|, i \in \vaset\}$. This is a norm on $\C^n$. We will write $B(\al,r)$ for the ball of radius $r$ centered at $\al \in \C^n$.
\end{defi}

\begin{defi} \label{defi:11} 
For a monomial operator $T$, we let $U_T=\{\z \in \C^n| \ T^k\z \longrightarrow 0\}$.
\end{defi}

The statement of the main result is the following.

\begin{prop} [Meromorphy on the unit ball] \label{prop:15} $\newline$
Let $T$ be in the class $\mathcal{F}$. Then any $M_T$-function extends as a meromorphic function on the domain $U_T$. In particular it extends as a meromorphic function on $B(0,1)$. 
\end{prop}

In order to prove this proposition, we first show that $B(0,1) \subseteq U_T$.

\begin{lem} \label{prop:13} 
If $T$ is in the class $\mathcal{F}$, then $B(0,1) \subseteq U_T$.    
\end{lem}

\begin{proof}
Let $\z \in B(0,1)$. Then for $a=||\z||<1$, we know that for all $i \in \vaset$, $|z_i|\leq a$. By the second point of \Cref{lem:expending}, the lengths of all $T^kz_i$ tend to $+\infty$ and so for all $p \in \N$, all coordinates of $T^k\z$ have modulus at most $a^p$ for $k$ large enough. This shows that $T^k\z \longrightarrow 0$.
\end{proof}

The next step is to prove that an $M_T$-function always extends as a meromorphic function on $U_T$. This is done via path continuation.

\begin{lem} \label{lem:star}
Let $T$ be in the class $\mathcal{F}$. Then $U_T$ is open and star-shaped around $\bm{0}$, and in particular simply connected.
\end{lem}

\begin{proof} $\newline$

\begin{enumerate}
    \item Let $\z \in U_T$. There is a $k \in \N$ such that $T^k\z \in B(0,1)$. Take $r>0$ such that $B(T^k\z,r) \subseteq B(0,1)$. Since $T^k$ is continuous, there exists $s>0$ such that for all $\z' \in B(\z,s)$, we have $T^k\z' \in B(T^k\z,r) \subseteq B(0,1)$.\\
Then for all $\z' \in B(\z,s)$, $T^m\z'=T^{m-k}T^k\z'$ tends to $0$ as $m$ tends to $+\infty$ by \Cref{prop:13}. This shows that $U_T$ is open.
\item Let $\z \in U_T$. For $t \in [0,1]$, let $\z_t=t\z=(tz_1,\ldots,tz_n)$. Then
$$T^m\z_t=(t^{k_{1,m}}(T^m\z)_1,\ldots,t^{k_{n,m}}(T^m\z)_n)$$
for some integers $k_{i,m} \geq 0$ (which are the sums of the terms in each column of $T^m$).
The modulus of each coordinate of $T^m\z_t$ is smaller than the modulus of the corresponding coordinate of $T^m\z$, so $\z_t \in U_T$ for all $t \in [0,1]$. This shows that $U_T$ is star-shaped around $0$, and therefore simply connected. 
\end{enumerate}
\end{proof}

We can now prove \Cref{prop:15}.

\begin{proof} [Proof of \Cref{prop:15}]  $\newline$
By \Cref{lem:star}, $U_T$ is star shaped around $0$ so it suffices to prove that for all $\z \in U_T$, $f$ can be continued meromorphically along the line segment which links $0$ and $\z$.\\
Let $\z \in U_T$ and consider the path given by $\gamma(t)=\z_t=t\z$ for $t \in [0,1]$. By \Cref{thm:5}, $f$ is analytic on a neighborhood of the origin $B(0,r)$ for some $0<r<1$. In particular, $f$ extends along $\gamma$ up to $t=\frac{r}{||\z||}$. Let $t_0 >0$ be the supremum of the $t \in [0,1]$ such that $f$ extends as a meromorphic function along the path $\gamma|_{[0,t]}$. Assume that $t_0 <1$. Then consider $\z_{t_0}$ which we know is in $ U_T$. By hypothesis, there is a $k \in \N$ such that $T^k\z_{t_0} \in B(0,r)$. Since $T^k$ is continuous, there is a neighborhood $V$ of $\z_{t_0}$ such that $T^kv \in B(0,r)$ for all $v \in V$.\\
Now since $f$ is $T$-Mahler, it is also $T^k$-Mahler and we can take a Mahler equation
$$\sum_{i=0}^m{p_i(\z)f((T^k)^i\z)}=0$$
with $p_0 \neq 0$, which gives
$$f(\z)=-\frac{1}{p_0(\z)}\sum_{i=1}^m{p_i(\z)f((T^k)^i\z)} \, .$$
If $\z \in V$, then $T^k\z \in B(0,r)$ and therefore $(T^{k})^i\z \in B(0,r)$ for all $i \in \llbracket 1,m \rrbracket$ (note that the ball $B(0,r)$ is stable under $T^k$ because the length of all coordinates of $T^k\z$ is at least $1$). Since $f$ is analytic on $B(0,r)$, the right-hand side of this equation is a meromorphic function on $V$, and therefore $f$ extends to a meromorphic function on $V$ which implies that $f$ extends as a meromorphic function along $\gamma|_{[0,t_1]}$ for some $t_1 >t_0$. This contradicts the maximality of $t_0$.\\
Therefore $t_0=1$ and $f$ extends as a meromorphic function along all line segments $\gamma$, which ends the proof.
\end{proof}

\subsection{Global analyticity and meromorphy} \label{sect:3.4} $\newline$

In this subsection, we prove that $M_T$-functions (with $T$ in the class $\mathcal{F}$) are meromorphic on the $p$-adic unit ball for every prime $p$.\\
Let $f$ be an $M_T$-function with $T$ in the class $\mathcal{F}$. There exists a number field $\K$ containing all the coefficients of $f$.\\
For all $p \in \mathbf{P} \cup \{\infty\}$, we fix $(\C_p,|\cdot|_p)$ a complete algebraically closed field containing $\Q_p$ (for $p=\infty$, we set $\C_{\infty}=\C$ by definition), and we fix an embedding $\sigma_p$ of $\K$ into $\C_p$. This way, the power series $f$ can be thought of as a power series with coefficients in $\C_p$, and therefore as the germ of a function from $\C_p^n$ to $\C_p$.\\

Next, we recall the definition of analytic and meromorphic functions in the $p$-adic setting (with $p<\infty$).\\
We first define the Tate algebra as the set of formal power series with coefficients in $\C_{p}$ which converge on the closed unit polydisk in $\C_{p}^n$, that is
$$A_{p}=\left\{\sum_{\bm{k} \in \N^n}{a(k_1,\ldots,k_n) z_1^{k_1}\ldots z_n^{k_n}} \in \C_{p}[[\z]] \Bigg{|} \lim_{|\bm{k}| \longrightarrow +\infty}{|a(\bm{k})|_p}=0 \right\} \, .$$
Furthermore, if $R \in \R_+^*$ we also let
$$A_{p}(R)=\left\{\sum_{\bm{k} \in \N^n}{a(k_1,\ldots,k_n) z_1^{k_1}\ldots z_n^{k_n}} \in \C_{p}[[\z]] \Bigg{|} \lim_{|\bm{k}| \longrightarrow +\infty}{|a(\bm{k})|_p \ R^{|\bm{k}|}}=0 \right\} \, .$$
We refer to \cite{bosch2014lectures} for an overview of this notion.\\
We will say that a power series $f$ with coefficients in $\C_p$ is analytic on the closed ball
$$\overline{B_p}(0,R)=\{\z \in \C_p^n| \ \forall i \in \vaset, |z_i|_p\leq R\}$$
if $f \in A_p(R)$, and meromorphic on $\overline{B_p}(0,R)$ if $f \in \Frac(A_p(R))$.\\
The closed polydisks $(\overline{B_p}(0,R), R<1)$ form an affinoid covering of the rigid space
$$B_p(0,1)=\{\z \in \C_p^n| \ \forall i \in \vaset, |z_i|_p< 1\} \, .$$
In accordance with the standard definition (see Section 4.6 of \cite{Rigid}), we will say that $f$ is meromorphic on $B_p(0,1)$ if $f$ is meromorphic on $\overline{B_p}(0,R)$ for all $R<1$.

\begin{rk} \label{rk:emb}
A different choice of embedding $\sigma_p$ of $\K$ into $\C_p$ would give a different notion of \enquote{meromorphic on the $p$-adic unit ball}.\\
However if $\sigma'_p$ is another such embedding, then we can extend both $\sigma_p$ and $\sigma'_p$ to isomorphisms between $\A$ and
$$\A^{(p)}=\{x \in \C_p \, | \, x \ \text{is algebraic} \} \, .$$
Then $\sigma=(\sigma_p)^{-1} \circ \sigma'_p$ is an isomorphism of $\A$. Thus, studying the power series
$$ f(z_1,...,z_n)=\sum_{k_1,\ldots,k_n \geq 0}{a(k_1,...k_n)z_1^{k_1}\ldots z_n^{k_n}}$$ under the embedding $\sigma_p'$ is the same as studying
$$f^{\sigma}(z_1,\ldots,z_n)=\sum_{k_1,\ldots,k_n \geq 0}{\sigma(a(k_1,...k_n))z_1^{k_1}\ldots z_n^{k_n}}$$
under the embedding $\sigma_p$.\\
If we let
$$\sum_{l=0}^m{p_l(\z)f(T^l\z)}=0$$
be a $T$-Mahler equation satisfied by $f$, then we can write
$$\sum_{l=0}^m{p_l^{\sigma}(\z)f^{\sigma}(T^l\z)}=0$$
and thus $f^{\sigma}$ is still an $M_T$-function.\\
Using this remark, the statement \enquote{all $M_T$-functions are meromorphic on the open unit ball in $\C_p^n$} actually does not depend on the embedding that is chosen. 
\end{rk}

The main result of this subsection is the following, which can be thought of as an extension of the main results of \Cref{sect:3.2} and \Cref{sect:3.3} to the global setting (looking at all places simultaneously).

\begin{prop} \label{prop:merop} 
Let $f$ be an $M_T$-function with $T$ in the class $\mathcal{F}$. Let $p \in \mathbf{P} \cup \{\infty\}$ and view $f$ as a power series with coefficients in $\C_p$ as above. Then using the definitions above, we have
\begin{enumerate}
    \item There exists an $r>0$ such that for all $p\in \mathbf{P} \cup \{\infty\}$, $f$ is analytic on the ball $\overline{B_p}(0,r)$.
    \item For all $p\in \mathbf{P} \cup \{\infty\}$, $f$ is meromorphic on $B_p(0,1)$.
\end{enumerate}
\end{prop}

For the proof of \Cref{prop:merop}, we first recall a standard result
\begin{lem} \label{lemma:value}
Let $p \in \mathbf{P}$. Then for all $R \in p^{\Q}$, there exists $x \in \C_{p}$ with $|x|_p=R$. 
\end{lem}

\begin{proof} [Proof of \Cref{lemma:value}]
Let $R=p^{\frac{a}{b}}$ with $a \in \Z$ and $b \in \N^*$. We know that $p \in \Z$ has absolute value $p^{-\frac{d_p}{d}}$ using the conventions at the start of \Cref{sect:3.1}. Taking $(bd_p)$-th roots in $\C_{p}$, we get that $\left|p^{\frac{-ad}{bd_p}}\right|_p=p^{\frac{a}{b}}=R$.
\end{proof}

We are ready to prove \Cref{prop:merop}.
\begin{proof} [Proof of \Cref{prop:merop}]
The first point is a reformulation of \Cref{thm:5} after taking $r<\frac{1}{C}$.\\

For the second point, we will focus on the case $p<\infty$, as the case $p=\infty$ was done in \Cref{prop:15} (also using \Cref{rk:emb} for different embeddings into $\C$).\\

Let $R<1$, and assume for now that $R \in p^{\Q}$. By \Cref{lem:expending}, there exists some $J \in \N$ such that for all $j \geq J$, the length of all coordinates in $T^j\z$ is greater than $L$, where $L$ is chosen big enough so that $R^L<r$. Now write a $T^J$-Mahler equation for $f$ in the form:
$$f(\z)=\frac{- \sum_{i=1}^m{q_i(\z)f(T^{Ji}\z)}}{q_0(\z)}=\frac{h(\z)}{q_0(\z)} \, .$$
By the previous lemma, we can find $x \in \C_{p}$ such that $|x|_p=R$. The numerator
$$h(\z)=- \sum_{i=1}^m{q_i(\z)f(T^{Ji}\z)}$$
converges at $\z_0=(x,x,\ldots,x)$, since every coordinate of $T^{Ji}\z_0$ has absolute value less than $r$. Therefore, the general term of that series at $\z_0=(x,x,\ldots,x)$ tends to $0$ for $|\cdot|_p$ (see Lemma 3 page 10 in \cite{bosch2014lectures}), which means exactly that $h \in A_{p}(R)$. Since $q_0(\z)$ is a polynomial, it is also in $A_{p}(R)$, and therefore $f \in \Frac(A_{p}(R))$.\\
In the general case, let $R<1$ not necessarily in $p^{\Q}$. Since $p^{\Q}$ is dense in $\R_+^*$, we can take an $R' \in p^{\Q}$ such that $R<R'<1$. Then we know that $$f\in \Frac(A_{p}(R')) \subseteq \Frac(A_{p}(R))$$
By definition, this shows that $f$ is meromorphic on $B_p(0,1)$.
\end{proof}

\section{Proofs of \Cref{thm:dic} and \Cref{thm:lift}}
The goal of this section is to prove the Lifting theorem (\Cref{thm:lift}). From Theorem 3.3 in \cite{adamczewski2020mahler}, we know that this theorem is true if the extension $$\A(\z)(f_1(\z),\ldots,f_n(\z))$$ is a regular extension of $\A(\z)$.\\
We start this section by proving \Cref{thm:dic}. The proof uses the result below which follows from Corollary 1.3 in \cite{bell2024d}. 

\begin{theo} \label{thm:bell}
Let $\K$ be a number field and let $S$ be a finite set of places of $\K$ containing the archimedean ones. Let
$$F(\z)=\sum_{\bm{k} \in \N^n}{f(\bm{k}) \z^{\bm{k}}}$$
be a power series in $n$ variables with coefficients in $\K$. We assume that
\begin{enumerate}
    \item $F$ is algebraic over $\A(\z)$.
    \item The $f(\bm{k})$ are all $S$-integers.
    \item The logarithmic Weil height of the coefficients satisfies $h(f(\bm{k}))=o(|\bm{k}|)$.
\end{enumerate}
Then $F$ is a rational function.
\end{theo}

\begin{rk}
Corollary 1.3 in \cite{bell2024d} assumes that $F$ is a $D$-finite power series instead of condition (1). Since algebraic powers series are $D$-finite (see \cite{Lips}), \Cref{thm:bell} is well implied by this result.
\end{rk}

We now have everything to prove \Cref{thm:dic}.

\begin{proof}[Proof of \Cref{thm:dic}] $\newline$
Let $f$ be an $M_T$-function and assume that it is algebraic over $\A(\z)$. There exists $Q \in \A[\z]$ such that $g=Qf$ is integral over $\A[\z]$. Then $g$ is still an $M_T$-function, hence it satisfies two equations
$$p_m(\z)g(T^m\z)+\ldots+p_1(\z)g(T\z)+p_0(\z)g(\z)=0$$
with $p_i \in \A[\z]$ and 
$$g(\z)^p+a_{p-1}(\z)g(\z)^{p-1}+\ldots+a_{1}(\z)g(\z)+a_{0}(\z)=0$$
with $a_j \in \A[\z]$.\\
The idea is to apply \Cref{thm:bell} above to $g$.\\

We let $\K$ be a number field containing all coefficients of $g$ along with all the coefficients of the $p_i$'s and the $a_j$'s. Then we know from \Cref{prop:field} that there exists a finite set $S$ of places of $\K$ (containing the archimedean ones) such that all coefficients of $g$ are $S$-integers. Furthermore, $g$ is algebraic over $\A(\z)$. Therefore we only need to prove that the coefficients $c(k_1,\ldots,k_n)$ of $g$ satisfy the height condition
$$h(c(\bm{k}))=o(|\bm{k}|)$$
where $|\bm{k}|=k_1+\ldots+k_n$ and $h$ is the logarithmic Weil height.\\
If we let $\mathcal{M}(\K)$ be the set of places of $\K$, then
$$h(c(\bm{k}))=\sum_{\nu \in \mathcal{M}(\K)}{\max(0, \log{|c(\bm{k})|_{\nu}})}=\sum_{\nu \in S}{\max(0, \log{|c(\bm{k})|_{\nu}})} \, .$$
We will use the notation $\log^+(x)=\max(0,\log(x))$ as a shorthand.\\
Since $S$ is finite, it suffices to prove that $\log^+|c(\bm{k})|_{\nu}=o(|\bm{k}|)$ for all $\nu \in S$. We separate the proof of that bound for archimedean and non-archimedean places.\\
\\
\underline{Archimedean places}:\\
An archimedean place is given by $|x|_{\sigma}=|\sigma(x)|^i$ for some embedding $\sigma: \K \longrightarrow \C$ and some $i \in \{1/d,2/d\}$ where $d=[\K:\Q]$. This $i$ is just a constant which does not change anything for what we need to prove. As explained in \Cref{rk:emb}, we can assume that $\sigma=id: \K \rightarrow \C$, up to considering a conjugate of $g$ (which will still be an $M_T$-function and integral over $\A[\z]$).\\ 

Since $g$ is an $M_T$-function, it is meromorphic on the open unit ball $B(0,1)$ in $\C^n$ (By \Cref{prop:15} and \Cref{prop:13}). Write $g(\z)=\frac{g_1(\z)}{g_2(\z)}$ where $g_1$ and $g_2$ are two analytic functions on $B(0,1)$. Let $M$ be the set of zeros of $g_2$ in $B(0,1)$.\\
We will apply the Riemann extension theorem (Theorem 3 page 19 in \cite{gunning2022analytic}). $M$ is a thin set in $B(0,1)$, $g$ is holomorphic in $B(0,1) \setminus M$, and the fact that $g$ satisfies a monic algebraic equation implies that it is locally bounded on $B(0,1)$. Indeed, if $\z \in B(0,1)$, and if we assume that $g$ is not locally bounded near $\z$, then there would exist a sequence $(\z_n)$ of points of $B(0,1) \setminus M$ such that $\z_n \longrightarrow \z$ and $|g(\z_n)| \longrightarrow +\infty$. Then writing
$$0=g(\z_n)^p+\sum_{i=0}^{p-1}{a_i(\z_n)g(\z_n)^i}=g(\z_n)^p \left(1+\sum_{i=0}^{p-1}{\frac{a_i(\z_n)}{g(\z_n)^{p-i}}} \right)$$
we see that the modulus of the right-hand side tends to $+\infty$ (since the $
a_i(\z_n)$ are bounded and $|g(\z_n)| \longrightarrow +\infty$), which gives a contradiction.\\

Thus by the Riemann extension theorem, $g$ extends to a holomorphic function (still denoted by $g$) on $B(0,1)$. Let $R<1$. Then by the Cauchy estimates on the polysphere of radius $R$ we get
$$|c(k_1,\ldots,k_n)|=\left| \frac{g^{(\bm{k})}(0)}{k_1! \ldots k_n!} \right| \leq \frac{\max_{|z_1|=\ldots=|z_n|=R}|g(\z)|}{R^{k_1+\ldots+k_n}}=\frac{M_R}{R^{k_1+\ldots+k_n}}$$
and therefore
$$\frac{\log |c(\bm{k})|}{|\bm{k}|} \leq \frac{\log(M_R)}{|\bm{k}|}-\log(R) \, .$$
Thus
$$\limsup_{|\bm{k}| \rightarrow +\infty}{\frac{\log |c(\bm{k})|}{|\bm{k}|}} \leq -\log(R) \, .$$
Since this is true for all $R<1$, this shows that 
$$\limsup_{|\bm{k}| \rightarrow +\infty}{\frac{\log |c(\bm{k})|}{|\bm{k}|}} \leq 0$$
which implies that
$$\frac{\log^+ |c(\bm{k})|}{|\bm{k}|} \longrightarrow 0 \, .$$
Thus, the bound $\log^+|c(\bm{k})|_{\nu}=o(|\bm{k}|)$ is proven for archimedean places $\nu$.\\
\\
\underline{Non-archimedean places}:\\
A non-archimedean place of $\K$ is given by $|\cdot|_{p,\sigma}=|\sigma(\cdot)|_p^i$ for some embedding $\sigma$ of $\K$ into $\C_p$. Again, the value of $i$ does not matter for what we need to prove, and we can assume that $\sigma=id: \K \rightarrow \C_p$ (where $\K$ is implicitly viewed as a subset of $\C_p$ using the fixed embedding $\sigma_p$, as explained in \Cref{sect:3.4}), up to considering a conjugate of $g$.\\
Recall the definitions of the Tate algebras $A_{p}(R)$ in section \Cref{sect:3.4}. According to Proposition 15 (page 20) in \cite{bosch2014lectures}, $A_{p}=A_p(1)$ is factorial, hence integrally closed in its field of fractions.\\
Let $R \in p^{\Q}$. According to \Cref{lemma:value}, there is an $x \in \C_p$ such that $|x|_p=R$. Then the map
$$F(z_1,\ldots,z_n) \longrightarrow F(xz_1,\ldots,xz_n)$$
gives a ring isomorphism between $A_{p}(R)$ and $A_{p}$. Thus, for all $R \in p^{\Q}$, $A_{p}(R)$ is also integrally closed in its field of fractions.\\

Let $R \in p^{\Q}$ with $R<1$. According to \Cref{prop:merop}, we know that $$g\in \Frac(A_{p}(R)) \, .$$
Now, $g$ satisfies a monic algebraic equation with polynomial coefficients (hence in $A_{p}(R)$), and this implies $g \in A_{p}(R)$ since $A_{p}(R)$ is integrally closed in $\Frac(A_{p}(R))$.\\
Therefore, we get
$$\lim_{|\bm{k}| \longrightarrow +\infty}{|c(k_1,\ldots,k_n)|_p\ R^{|\bm{k}|}}=0$$
hence $|c(\bm{k})|_{p}=o((R^{-1})^{|\bm{k}|})$ for all $R<1$ such that $R \in p^{\Q}$. Since $p^{\Q}$ is dense in $\R_+^{*}$, this means that $|c(\bm{k})|_p=o((1+\varepsilon)^{|\bm{k}|})$ for all $\varepsilon >0$, hence we obtain $$\log^+|c(\bm{k})|_{p}=o(|\bm{k}|) \, .$$
This shows that the bound $\log^+|c(\bm{k})|_{\nu}=o(|\bm{k}|)$ is also true for non-archimedean places $\nu$.\\

Finally, we can apply \Cref{thm:bell} which tells us that $g$ is a rational function, and therefore $f$ is also a rational function, which ends the proof.
\end{proof}

Once \Cref{thm:dic} is proven, it is relatively easy to obtain the regularity of extensions.

\begin{prop} [Regularity of extensions] \label{prop:16} $\newline$
Let $T$ be in the class $\mathcal{F}$. Let $f_1, \ldots, f_m$ be $M_T$-functions, then the field extension $$\A(\z)(f_1(\z),\ldots,f_m(\z))$$ is a regular extension of $\A(\z)$.
\end{prop}

\begin{proof}[Proof of \Cref{prop:16}] $\newline$
Let $\mathbb{L}=\A(\z)(T^lf_i, 1 \leq i\leq m, l \in \N)$. Then it suffices to prove that $\mathbb{L}$ is a regular extension of $\A(\z)$. Let $\K$ be the algebraic closure of $\A(\z)$ in $\mathbb{L}$. Then since the $f_i$ are $M_T$-functions, $\mathbb{L}$ is finitely generated over $\A(\z)$, and so $\K$ is finitely generated over $\A(\z)$ (it is a general algebraic result that a subextension of a finitely generated extension is itself finitely generated: see \cite{lang2012algebra}, Exercise 4, Chapter 8). But since $\K$ is an algebraic extension of $\A(\z)$, it is a finite extension of $\A(\z)$ (say of degree $d$).\\
Let $f \in \K$. Then $f, Tf, \ldots, T^df$ are all algebraic over $\A(\z)$ by applying $T^k$ to the algebraic equation satisfied by $f$. Therefore they are all in $\K$, and so they must be linearly dependent over $\A(\z)$. This means that $f$ satisfies a $T$-Mahler equation. We cannot conclude directly that $f \in \A(\z)$ since we do not know if $f$ is a power series. However, we do know that $f \in \mathbb{L}$ so $f \in \Frac(\A[[\z]])$.\\
Since $f$ is algebraic over $\A(\z)$, we know that there exists $Q \in \A[\z]$ such that $g=Qf$ is integral over $\A[\z]$, hence also over $\A[[\z]]$. But $g \in \Frac(\A[[\z]])$ and $\A[[\z]]$ is factorial and therefore also integrally closed in its field of fractions. Therefore, $g \in \A[[\z]]$. Since $g$ still satisfies a Mahler equation, it is an $M_T$-function. But since it is also still algebraic over $\A(\z)$, we obtain $g \in \A(\z)$ from \Cref{thm:dic}, and therefore $f \in \A(\z)$.\\
We can conclude that $\K=\A(\z)$, and therefore $\mathbb{L}$ is a regular extension of $\A(\z)$, which completes the proof.
\end{proof}

We end this section with the proof of \Cref{thm:lift}.

\begin{proof} [Proof of \Cref{thm:lift}] $\newline$
As we explained at the beginning of this section, Theorem 3.3 in \cite{adamczewski2020mahler} already shows that \Cref{thm:lift} is true when we assume that $\A(\z,f_1(\z),\ldots,f_m(\z))$ is a regular extension of $\A(\z)$. We showed in the previous proposition that this condition is always satisfied, hence \Cref{thm:lift} is true unconditionally.
\end{proof}

\section{The descent theorem at regular points}

We start this section with the proof of \Cref{coro:trans}. Then we prove a descent theorem at regular points.

\begin{proof} [Proof of \Cref{coro:trans}] $\newline$
We know that $\al$ is a regular point with respect to the companion system associated to the Mahler equation, and that $(T,\al)$ is admissible. Furthermore, the functions\\
$(f,Tf,\ldots,T^{m-1}f,1)$ are $\A(\z)$-linearly independent by minimality of $m$. Therefore the values $(f(\al),Tf(\al),\ldots,T^{m-1}f(\al),1)$ are $\A$-linearly independent.\\
If $f(\al)=\beta$ were algebraic, there would exist a $\A$-linear relation $1 \cdot f(\al)-\beta \cdot 1=0$ between the values at $\al$ of $f$ and $g=1$. This gives a contradiction.\\
\end{proof}

For this corollary to apply, we need to prove that $(f,Tf,\ldots,T^{m-1}f,1)$ is linearly independent over $\A(\z)$, which can be difficult to do (however, we can be sure that there exists some $m$ such that this is true). To avoid this problem, one method that is frequently used in the univariate case is to use the following \enquote{Descent theorem}: if $f$ is an $M_q$-function (with $q\ge2)$ with coefficients in a number field $\K$, and if $\alpha \in \K$ with $|\alpha|<1$, then either $f(\alpha) \in \K$ or $f(\alpha)$ is transcendental. We prove a particular case of this theorem for multivariate $M_T$-functions.

\begin{theo} [Descent Theorem at regular points] \label{thm:9} $\newline$
Let $Y=(f_1,\ldots,f_m)^*$ be a vector of $M_T$-functions with coefficients in a number field $\K$, related by a system $Y(\z)=A(\z)Y(T\z)$ with $A \in \GL_n(\A(\z))$. We assume that $\al \in (\K^*)^n$ is a regular point with respect to this system and that $(T,\al)$ is admissible.\\
Let $I \subseteq \llbracket 1,m \rrbracket$ be any subset. Then the family $(f_j(\al))_{j \in I}$ is linearly independent over $\A$ if and only if it is linearly independent over $\K$.\\
In particular, for all $j \in \llbracket 1, m \rrbracket$, either $f_j(\al)$ is in $\K$ or it is transcendental.\\
Furthermore, if $(b_1,\ldots,b_m) \in \K^m$, then $b_1f_1(\al)+\ldots+b_mf_m(\al)$ is either in $\K$ or transcendental.
\end{theo}

In order to prove \Cref{thm:9}, we first show a descent result at a functional level.

\begin{lem} \label{lem:funcdes}
Let $h_1(\z), \ldots, h_m(\z)$ be analytic functions in a neighborhood of the origin with coefficients in a number field $\K$. Let $\al \in \K^n$ be a point in the domain of convergence of these functions. Let $J \subseteq \llbracket 1,m \rrbracket $ be any subset.\\
Assume that we have a linear dependence relation
$$w_1(\z)h_1(\z)+\ldots+w_m(\z)h_m(\z)=0$$
where $w_i \in \A[\z]$ are not all zero, and such that $w_j(\al)=0$ for all $j \in J$, and $w_{i_0}(\al) \neq 0$.\\
Then there exist polynomials $w_i' \in \K[\z]$ not all zero, such that
$$w_1'(\z)h_1(\z)+\ldots+w_m'(\z)h_m(\z)=0$$
with $w_j'(\al)=0$ for all $j \in J$, and $w_{i_0}'(\al) \neq 0$.
\end{lem}

\begin{proof} [Proof of \Cref{lem:funcdes}] $\newline$
The proof in \cite{adamczewski2017methode} (Lemma 5.3) directly generalizes to the multivariate case. We recall the argument below for convenience.\\
We use that the coefficients of the polynomials $w_i$ all lie in a common number field $\K'$, which we can assume to contain $\K$. The extension $\K'/\K$ is finite so it is generated by a primitive element $\varphi \in \K'$. This gives a $\K$-linear isomorphism $f=(f_1,\ldots,f_d):\K' \simeq \K^d$ with $d$ the degree of $\varphi$ over $\K$, which sends $x \in \K'$ to $(x_1,\ldots,x_d)$ such that $x=x_1+x_2 \varphi+\ldots+x_d\varphi^{d-1}$. This $f$ naturally extends to $\K'[\z]$ by applying $f$ to all coefficients.\\
Since $w_{i_0}(\al) \neq 0$, there is an $l$ such that $f_l(w_{i_0}(\al)) \neq 0$. Then we let $w_i'(\z)=f_l(w_i(\z))$ for all $i$. Since $f_l$ is $\K$-linear and the $h_i$ have coefficients in $\K$, we obtain 
$$w_1'(\z)h_1(\z)+\ldots+w_m'(\z)h_m(\z)=0$$
by looking at each coefficient.\\
Furthermore, we have $w_j'(\al)=f_l(w_{j}(\al))=0$ for all $j \in J$, and $w_{i_0}'(\al)=f_l(w_{i_0}(\al)) \neq 0$. Note that it is also obvious that $w_{i_0}'(\z) \neq 0$ since $w_{i_0}'(\al) \neq 0$, therefore the new relation is non-trivial.
\end{proof}

Thanks to the previous lemma and the lifting theorem, we can give a proof of \Cref{thm:9}.

\begin{proof} [Proof of \Cref{thm:9}] $\newline$
For the first part, we assume that $(f_j(\al))_{j \in I}$ is linearly dependent over $\A$ and we want to prove that it is linearly dependent over $\K$ (the other direction is always true).\\
Let $(\lambda_j)_{j \in I}$ be algebraic numbers not all zero such that
$$\sum_{j \in I}{\lambda_j f_j(\al)}=0 \, .$$
We can view that as a linear relation between all the $f_i(\al)$ for $i \in \llbracket 1,m \rrbracket$. By \Cref{thm:lift} for homogeneous relations of degree 1, we obtain that the $f_i(\z)$ satisfy a linear relation with coefficients in $\A[\z]$
$$w_1(\z)f_1(\z)+\ldots+w_m(\z)f_m(\z)=0$$
with $w_j(\al)=\lambda_j$ for $j \in I$ and $w_j(\al)=0$ for $j \notin I$. Now we apply the previous lemma with $J=\llbracket 1,m \rrbracket \setminus I$ and with $i_0 \in I$ such that $\lambda_{i_0} \neq 0$. It gives polynomials $w_i' \in \K[\z]$ such that
$$w_1'(\z)f_1(\z)+\ldots+w_m'(\z)f_m(\z)=0$$
which vanish at $\al$ for $i \notin I$ and such that $w'_{i_0}(\al) \neq 0$. Setting $\z=\al$ in the previous relation gives
$$\sum_{j \in I}{w_j'(\al)f_j(\al)}=0$$
which is a non trivial $\K$-linear relation between the $(f_j(\al))_{j \in I}$. This proves the first point.\\
\\
For the second point, we just set $f_{m+1}=1$. Then $\tilde{Y}=(f_1,\ldots,f_{m+1})$ satisfies the system associated to the matrix $B(\z)=\Diag(A(\z),1)$. It is clear that $\al$ is still a regular point for this system, therefore we can apply the first part with $I=\{j,m+1\}$. This gives that $(f_j(\al),1)$ is linearly independent over $\A$ if and only if it is linearly independent over $\K$. This means exactly that $f_j(\al)$ is either in $\K$ or transcendental.\\
\\
For the third point, we can assume that $(b_1,\ldots,b_m) \in \K^m \setminus \{0\}$. We let $B \in \GL_m(\K)$ such that the first row of $B$ is $(b_1,\ldots,b_m)$. Then we perform the change of variables $Y_1=BY$. We have $$Y_1(\z)=BY(\z)=BA(\z)Y(T\z)=BA(\z)B^{-1}Y_1(T\z) \, .$$
Thus $Y_1$ still satisfies a Mahler system with matrix $BA(\z)B^{-1}$.\\
$B$ has no pole, and conjugation does not change the determinant, so $\al$ is still regular with respect to this new system. Applying the second point to the first coordinate of $Y_1$ (which is $b_1f_1+\ldots+b_mf_m$) gives the dichotomy that we wanted.
\end{proof}

\nocite{*}
\bibliographystyle{plain}
\bibliography{biblio}

\begin{thebibliography}{10}

\bibitem{Gap}
B.~Adamczewski, J.~Bell, and D.~Smertnig.
\newblock A height gap theorem for coefficients of {M}ahler functions.
\newblock {\em J. Eur. Math. Soc. (JEMS)}, 25(7):2525--2571, 2023.

\bibitem{Hypertrans}
B.~Adamczewski, T.~Dreyfus, and C.~Hardouin.
\newblock Hypertranscendence and linear difference equations.
\newblock {\em J. Amer. Math. Soc.}, 34(2):475--503, 2021.

\bibitem{adamczewski2017methode}
B.~Adamczewski and C.~Faverjon.
\newblock M{\'e}thode de mahler: relations lin{\'e}aires, transcendance et applications aux nombres automatiques.
\newblock {\em Proceedings of the London Mathematical Society}, 115(1):55--90, 2017.

\bibitem{Algotrans}
B.~Adamczewski and C.~Faverjon.
\newblock M\'ethode de {M}ahler, transcendance et relations lin\'eaires: aspects effectifs.
\newblock {\em J. Th\'eor. Nombres Bordeaux}, 30(2):557--573, 2018.

\bibitem{E+M}
B.~Adamczewski and C.~Faverjon.
\newblock Relations alg\'ebriques entre valeurs de {$E$}-fonctions ou de {$M$}-fonctions.
\newblock {\em C. R. Math. Acad. Sci. Paris}, 362:1215--1241, 2024.

\bibitem{adamczewski2020mahler}
B.~Adamczewski and C.~Faverjon.
\newblock Mahler's method in several variables and finite automata.
\newblock {\em Annals of Mathematics}, 2026, to appear.

\bibitem{bell2024d}
J.~Bell, S.~Chen, K.~Nguyen, and U.~Zannier.
\newblock D-finiteness, rationality, and height {III}: multivariate {P}\'olya-{C}arlson dichotomy.
\newblock {\em Math. Z.}, 306(4):Paper No. 70, 13, 2024.

\bibitem{Beu}
F.~Beukers.
\newblock A refined version of the {S}iegel-{S}hidlovskii theorem.
\newblock {\em Ann. of Math. (2)}, 163(1):369--379, 2006.

\bibitem{bosch2014lectures}
S.~Bosch.
\newblock {\em Lectures on formal and rigid geometry}, volume 2105.
\newblock Springer, 2014.

\bibitem{CompSol}
F.~Chyzak, T.~Dreyfus, P.~Dumas, and M.~Mezzarobba.
\newblock Computing solutions of linear {M}ahler equations.
\newblock {\em Math. Comp.}, 87(314):2977--3021, 2018.

\bibitem{cobham1968hartmanis}
A.~Cobham.
\newblock On the hartmanis-stearns problem for a class of tag machines.
\newblock In {\em 9th Annual Symposium on Switching and Automata Theory (SWAT 1968)}, pages 51--60. IEEE, 1968.

\bibitem{Dumas}
P.~Dumas.
\newblock {\em R\'ecurrences mahl\'eriennes, suites automatiques, \'etudes asymptotiques}.
\newblock Institut National de Recherche en Informatique et en Automatique (INRIA), Rocquencourt, 1993.
\newblock Th\`ese, Universit\'e{} de Bordeaux I.

\bibitem{CompHahn}
C.~Faverjon and J.~Roques.
\newblock Hahn series and {M}ahler equations: algorithmic aspects.
\newblock {\em J. Lond. Math. Soc. (2)}, 110(1):Paper No. e12945, 60, 2024.

\bibitem{Rigid}
J.~Fresnel and M.~van~der Put.
\newblock {\em Rigid analytic geometry and its applications}, volume 218 of {\em Progress in Mathematics}.
\newblock Birkh\"auser Boston, Inc., Boston, MA, 2004.

\bibitem{gunning2022analytic}
R.~Gunning and H.~Rossi.
\newblock {\em Analytic functions of several complex variables}, volume 368.
\newblock American Mathematical Society, 2022.

\bibitem{Kubota}
K.~K. Kubota.
\newblock On the algebraic independence of holomorphic solutions of certain functional equations and their values.
\newblock {\em Math. Ann.}, 227(1):9--50, 1977.

\bibitem{lang2012algebra}
S.~Lang.
\newblock {\em Algebra}.
\newblock Springer Science \& Business Media, 2012.

\bibitem{Lips}
L.~Lipshitz.
\newblock {$D$}-finite power series.
\newblock {\em J. Algebra}, 122(2):353--373, 1989.

\bibitem{LVDP1}
J.~H. Loxton and A.~J. Van~der Poorten.
\newblock Arithmetic properties of certain functions in several variables.
\newblock {\em J. Number Theory}, 9(1):87--106, 1977.

\bibitem{LVDP2}
J.~H. Loxton and A.~J. Van~der Poorten.
\newblock Arithmetic properties of certain functions in several variables. {II}.
\newblock {\em J. Austral. Math. Soc. Ser. A}, 24(4):393--408, 1977.

\bibitem{LVDP3}
J.~H. Loxton and A.~J. van~der Poorten.
\newblock Arithmetic properties of certain functions in several variables. {III}.
\newblock {\em Bull. Austral. Math. Soc.}, 16(1):15--47, 1977.

\bibitem{mahler1929arithmetische}
K.~Mahler.
\newblock Arithmetische eigenschaften der l{\"o}sungen einer klasse von funktionalgleichungen.
\newblock {\em Mathematische Annalen}, 101(1):342--366, 1929.

\bibitem{mahler1930arithmetische}
K.~Mahler.
\newblock Arithmetische eigenschaften einer klasse transzendental-transzendenter funktionen.
\newblock {\em Mathematische Zeitschrift}, 32(1):545--585, 1930.

\bibitem{mahler1930uber}
K.~Mahler.
\newblock Uber das verschwinden von potenzreihen mehrerer ver{\"a}nderlichen in speziellen punktfolgen.
\newblock {\em Mathematische Annalen}, 103(1):573--587, 1930.

\bibitem{Masser}
D.~W. Masser.
\newblock A vanishing theorem for power series.
\newblock {\em Invent. Math.}, 67(2):275--296, 1982.

\bibitem{Nish1983}
Ku. Nishioka.
\newblock On a problem of {M}ahler for transcendency of function values. {II}.
\newblock {\em Tsukuba J. Math.}, 7(2):265--279, 1983.

\bibitem{Nish}
Ku. Nishioka.
\newblock New approach in {M}ahler's method.
\newblock {\em J. Reine Angew. Math.}, 407:202--219, 1990.

\bibitem{nishioka2006mahler}
Ku. Nishioka.
\newblock {\em Mahler functions and transcendence}.
\newblock Springer, 1996.

\bibitem{Phil}
P.~Philippon.
\newblock Groupes de {G}alois et nombres automatiques.
\newblock {\em J. Lond. Math. Soc. (2)}, 92(3):596--614, 2015.

\bibitem{Rand}
B.~Randé.
\newblock {\em Equation fonctionnelle de {M}ahler et application aux suites p-régulières}.
\newblock PhD thesis, Université {B}ordeaux I, 1992.

\bibitem{waldschmidt2000diophantine}
M.~Waldschmidt.
\newblock {\em Diophantine approximation on linear algebraic groups: transcendence properties of the exponential function in several variables}, volume 326.
\newblock Springer, 2000.

\end{thebibliography}
\end{document}